%% file: main.tex
\documentclass[11pt]{article}
\usepackage{pedro}
\usepackage{scalerel}
\usepackage{xcolor}
\usepackage{tikz}

\usetikzlibrary{arrows.meta,decorations.pathreplacing}

\DeclareMathOperator{\Tr}{\text{Tr}}

\include{shortcuts}

\date{}

\begin{document}
\title{Schrijver Number Quasi-Tensorization and Multicolor Ramsey Bounds via Robust OR Polynomials}

\author{
  Ijay Narang\thanks{Georgia Institute of Technology, School of Computer Science. \texttt{inarang3@gatech.edu}.}
  \and
  Yukai Tang\thanks{Princeton University, Operations Research and Financial Engineering.\texttt{yt3846@princeton.edu}.}
}

\date{\today}

\maketitle

\begin{abstract}
    We introduce a robust OR polynomial framework for composing positive semidefinite certificates across OR constraints. We demonstrate the power of this method in two applications.
    
The first is on acute-free families. A set $\mathcal F=\{(x_i^{(1)},\ldots,x_i^{(r)})\}_{i=1}^M
\subseteq (S^{n-1})^r$ is $r$-way acute-free if, for every $i\neq j$,
there is a coordinate $t\in[r]$ such that
$\langle x_i^{(t)},x_j^{(t)}\rangle\leq 0$. We write $M_r(n)$ for the maximum size of such a set, and
$M_r^{\pm}(n)$ for the hypercube restriction. On the hypercube, $r$-way acute-free sets are independent sets for some strong power graph $G_n^{\boxtimes r}$. The Lov\'asz theta number $\vartheta(G_n)$ is
multiplicative but
exponentially loose, whereas the Schrijver number $\vartheta'(G_n)$ gives the correct order, but is not multiplicative. We bypass this obstruction by proving a general quasi-tensorization result for the Schrijver number. That is, for every collection of graphs $G_1,\ldots, G_r$ satisfying $\vartheta'(G_i)\geq 2$, there is an absolute constant $C$ such that $\vartheta'(G_1\boxtimes \cdots \boxtimes G_r) \leq \prod_{i=1}^r \vartheta'(G_i)^{C\log r \log \vartheta'(G_i)}$. Applying this result gives that $M_r^{\pm}(n) \le M_r(n) \le (2n)^{C_0 r\log r\log(2n)}$ for some absolute constant $C_0$.

The second application is on multicolor Ramsey numbers. The $r$-color Ramsey number $R_r(k)$ is the minimum $n$ such that every
$r$-coloring of the edges of the complete graph on $n$ vertices contains a monochromatic copy of $K_k$. In a breakthrough result, \cite{balister2024upperboundsmulticolorramsey} showed that
\(R_r(k)\le \exp(-\Omega(k/r^{12}))r^{rk}\)
via a geometric lemma. By improving the $r$ dependency in their geometric lemma via the OR polynomial framework, we prove that
\(R_r(k)\le \exp(-\Omega(k/(r^9(\log r)^6)))r^{rk}\).
\end{abstract}

\section{Introduction}

Many problems in extremal combinatorics have an OR-type structure: one must certify that among several coordinates, at least one satisfies a desired property. While these problems are often easy to state combinatorially, they are harder to certify analytically. In contrast, one-coordinate analogues frequently admit certificates in the form of positive semidefinite matrices or kernels, as in~\cite{Lovasz1979,schrijver1986theory,delsarte1973algebraic,DeKlerkPasechnik2007,bachoc2009optimality}. The main challenge is then to translate these one-coordinate certificates to the OR setting. We illustrate this difficulty through a concrete example.

Let $\sphere$ denote the unit sphere in $\R^n$. For an integer \(r\ge 1\), we say that a set of \(r\)-tuples $\{(x_i^{(1)},\ldots,x_i^{(r)})\}_{i=1}^M \subseteq ( \sphere)^r$ is \emph{\(r\)-way acute-free} if for every \(i\ne j\), there exists $t\in[r]$ such that $\langle x_i^{(t)},x_j^{(t)}\rangle\leq 0.$  Given any positive integers $n$ and $r$, we ask: What is the maximum size, $M_r(n)$, of an $r$-way acute-free family?

Beyond being a natural OR-type problem, this problem has two motivations. The first comes from a recent work of ~\cite{balister2024upperboundsmulticolorramsey}.\footnote{In fact, the problem of characterizing $M_2(n)$ was explicitly asked by the authors of \cite{balister2024upperboundsmulticolorramsey} in presentations. We thank Noga Alon
for communicating this problem to us.} Their improved multicolor Ramsey
bound relies on a geometric lemma showing that joint restrictions on pairwise
inner products of several vector-valued functions force clustering in one
coordinate. A hypercube specialization of this theme led to a characterization
of worst-case inner products under bijections~\cite{NarangJu2025}, which asks
how much a bijection of \(\{\pm1\}^n\) can reduce simultaneous nonnegative
inner products. The problem of characterizing $M_r(n)$ is the corresponding extremal version: instead of
lower-bounding the probability of such a simultaneous event, we ask how large
a family can be when it is forbidden for every pair.

The second motivation comes from generalizing coding theoretic bounds. When $r = 1$, a classical theorem of Rankin states that a set of vectors \(A\subseteq \sphere\) with pairwise nonpositive inner products has size at most \(2n\), with equality attained by the cross-polytope \(\{\pm e_1,\ldots,\pm e_n\}\)~\cite{Rankin1955}, i.e., spherical codes at angular separation \(\pi/2\) have maximum size $2n$. In the hypercube setting, the corresponding threshold statement is the Plotkin bound for binary codes of relative distance at least \(1/2\) \cite{MacWilliamsSloane1977,vanLint1999}. 

To understand the additional difficulty of studying $M_r(n)$ versus $M_1(n)$, it is instructive to study the restricted case where all vectors lie on the Boolean hypercube. In this setting, analyzing the Boolean-restricted extremal quantity can be reformulated as a question about graph strong powers. 
Recall that the strong product \(G\boxtimes H\) of two graphs \(G\) and \(H\) is the graph with vertex set \(V(G)\times V(H)\), where two distinct vertices \((g,h)\) and \((g',h')\) are adjacent if
\[
(g=g' \text{ or } gg'\in E(G))
\qquad\text{and}\qquad
(h=h' \text{ or } hh'\in E(H)).
\]
and the \(r\)-fold strong power \(G^{\boxtimes r}\) is $\underbrace{G \boxtimes G \boxtimes \cdots \boxtimes G}_{r\text{ times}}$.

Let $\Hyp = \{\pm 1\}^n$ and $M_r^{\pm}(n)$ denote the largest size of an $r$-way acute-free set contained in $\Hyp^r$ and $G_n$ be the graph with vertex set $V(G_n)=\{\pm 1\}^n$, where two distinct vertices $x,y$ are adjacent if $\langle x,y\rangle>0$. Consider the $r$-th strong power $G_n^{\boxtimes r}$, whose vertices are $r$-tuples $(x^{(1)},\ldots,x^{(r)})\in(\{\pm 1\}^n)^r$. Two distinct $r$-tuples $(x^{(1)},\ldots,x^{(r)}),(y^{(1)},\ldots,y^{(r)})$ are adjacent in $G_n^{\boxtimes r}$ when they are acute in every coordinate, i.e., when $\langle x^{(t)},y^{(t)}\rangle>0$ for every $t\in[r]$. Therefore, an independent set in $G_n^{\boxtimes r}$ is an $r$-way acute-free set on the Boolean hypercube, giving that $M_r^{\pm}(n)=\alpha(G_n^{\boxtimes r}).$

One natural approach to bound $M_r^{\pm}(n)$ for any $r$ is using the \lovasz theta number $\vartheta(G)$. This is because the theta number enjoys the property that $\vartheta(G^{\boxtimes r}) =\vartheta(G)^r$. For a given graph $G = (V,E)$, it is defined as
\begin{equation}
    \vartheta(G)
    =
    \max \left\{
        \sum_{u,v\in V} X_{uv}
        : 
        X \succeq 0, \ 
        \operatorname{Tr}(X)=1,\ 
        X_{uv}=0 \quad \forall\, \{u,v\}\in E
    \right\}.
\end{equation} 

Unfortunately, for the graph \(G_n\) defined above, the \lovasz number displays a large gap. Samorodnitsky observed in an unpublished manuscript that \(\vartheta(G_n) \approx 2^{0.19n}\)~\cite{samorodnitsky1998extremal}, despite the fact that $\alpha(G_n) \leq 2n$ by the Plotkin bound. However, a strengthened semidefinite relaxation, called the Schrijver SDP \cite{Schrijver1979}, gives a much stronger bound on $G_n$. Given a graph $G = (V,E)$, we denote its Schrijver number to be $\vartheta'(G)$, which is defined as
\begin{equation}\label{def:schrijver-sdp-primal}
    \vartheta'(G)
    =
    \max \left\{
        \sum_{u,v\in V} X_{uv}
        :
        X \succeq 0,\ 
        X_{uv}\geq 0 \quad \forall\, u,v\in V,\ 
        \operatorname{Tr}(X)=1,\ 
        X_{uv}=0 \quad \forall\, \{u,v\}\in E
    \right\}.
\end{equation}
In fact, one always has $\vartheta'(G_n)\leq 2n$, which is significantly tighter than $\vartheta(G_n)$.\footnote{To the best of our knowledge, $G_n$ displays the largest known multiplicative separation between the \lovasz theta number $\vartheta$ and the Schrijver number $\vartheta'$.} One underlying reason for the tightness of $\vartheta'(G_n)$ is that the Schrijver number reduces to the Delsarte LP bound on association schemes \cite{Schrijver1979}. For more details on the Delsarte LP, see, for example, \cite{delsarte1973algebraic, godsil2016}. Unfortunately, despite only differing from $\vartheta(G)$ by the additional non-negativity constraint,  \(\vartheta'(G)\) is not multiplicative under graph strong product. Thus, to use the extra strength of \(\vartheta'(G_n)\), one may wish to also `tensorize' \(\vartheta'(G_n)\) to control \(\vartheta'(G_n^{\boxtimes r})\). Our first result achieves this by providing a general quasi-tensorization theorem for the Schrijver number. Before we state our result, we introduce the robust OR polynomial framework utilized to attain this result. 

At a high level, the framework composes one-coordinate positive semidefinite certificates into certificates for OR-type constraints. The starting point is a collection of one-coordinate positive semidefinite certificates. These certificates are first passed through compressors, which regularize their values so that any coordinate witnessing the OR constraint, namely any ``bad'' coordinate, is mapped close to \(-1\), while the remaining values stay in a controlled range. We then apply a robust OR polynomial with nonnegative coefficients. The robust OR property detects the presence of a bad coordinate, while coefficient nonnegativity ensures that the certificate remains positive semidefinite under composition. Thus, the method gives a general certificate-level way to handle OR constraints without requiring exact tensorization.

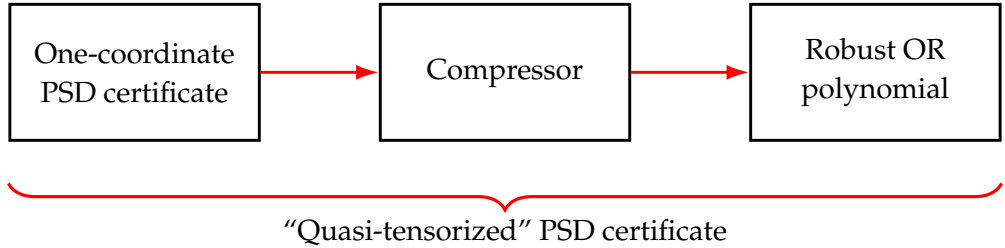
\begin{figure}[h]
\centering

\begin{tikzpicture}[
    box/.style={
        draw=black,
        line width=1.1pt,
        minimum width=3.3cm,
        minimum height=1.8cm,
        align=center,
        inner sep=8pt,
        font=\normalsize
    },
    arrow/.style={
        -{Latex[length=3mm,width=2mm]},
        line width=1.1pt,
        red
    },
    bigbrace/.style={
        decorate,
        decoration={brace, mirror, amplitude=10pt},
        line width=1.2pt,
        red
    }
]

\node[box] (A) at (0,0) {One-coordinate\\PSD certificate};
\node[box] (B) at (4.9,0) {Compressor};
\node[box] (C) at (9.8,0) {Robust OR\\polynomial};

\draw[arrow] (A.east) -- (B.west);
\draw[arrow] (B.east) -- (C.west);

\draw[bigbrace]
    ([yshift=-0.55cm]A.south west) --
    ([yshift=-0.55cm]C.south east)
    node[midway, yshift=-0.65cm, black, font=\normalsize]
    {``Quasi-tensorized'' PSD certificate};

\end{tikzpicture}

\caption{Diagrammatic Workflow of the robust OR polynomial framework.}
\label{fig:or-polynomial-proof-strategy}

\end{figure}

Instantiating this framework with an appropriately chosen one-coordinate kernel and compressor enables our first main theorem.

\begin{theorem} \label{thm:schrijver-quasi-tensorization} Let $r\geq 2$ and $G_1 = (V_1, E_1), \ldots, G_r = (V_r, E_r)$ be graphs. Suppose $\vartheta'(G_i) \geq 2$ for all $i\in [r]$. Then, there is an absolute constant $C > 0$ such that
\[
    \vartheta'(G_1\boxtimes \ldots \boxtimes G_r) \leq \prod_{i=1}^r \vartheta'(G_i)^{C\log r \log \vartheta'(G_i)}.
\]
\end{theorem}

The dependencies on $\vartheta'(G_i)$ and $r$ are discussed in Remark~\ref{rem:quasipolynomial-dependency}. 
A direct application of this quasi-tensorization theorem
gives quasipolynomial bounds on $M_r(n)$ and $M_r^{\pm}(n)$.

\begin{corollary} \label{cor:Mr(n)}
For every \(n\ge2\) and \(r\ge2\), let \(G_n\) be defined as above. Then there
exists a constant \(C_0>0\) such that
\begin{itemize}
    \item $M_r^{\pm}(n)
    \le
    \vartheta'(G_n^{\boxtimes r})
    \le
    (2n)^{C_0 r\log r\log(2n)}$.
    \item $M_r(n) \le
    (2n)^{C_0 r\log r\log(2n)}$.
\end{itemize}
\end{corollary}
We note that one can easily find a lower bound. Consider the set $A=\{\pm e_1,\ldots,\pm e_n\}\subset \mathbb R^n$, then
$A^r\subseteq (\mathcal S^{n-1})^r$ is $r$-way acute-free: any two distinct $r$-tuples differ in some coordinate, and in that coordinate
the corresponding vectors have nonpositive inner product. This, together with Corollary \ref{cor:Mr(n)} gives
\[
    (2n)^r \le M_r(n) \le (2n)^{C_0 r\log r\log(2n)}.
\]

We now proceed to our second application: improving upper bounds on multicolor Ramsey numbers via the robust OR polynomial framework. Let \(R_r(k)\) denote the diagonal \(r\)-color Ramsey number, namely the
smallest \(N\) such that every \(r\)-coloring of the edges of \(K_N\) contains a monochromatic \(K_k\). Recently, in a breakthrough result, \cite{balister2024upperboundsmulticolorramsey} proved that \(R_r(k)\le \exp(-\Omega(k/r^{12}))r^{rk}\).
The key to their analysis is a geometric lemma characterizing non-negative inner products for arbitrary
vector-valued functions. By applying our robust OR polynomial framework, we are able to prove a stronger geometric lemma in terms of the dependence on $r$. 

\begin{lemma}
\label{lem:or-composed-geometric-lemma}
Let \(r\ge 2\). Let \(U,U'\) be independent identically distributed random
variables taking values in a finite set \(X\), and let
\(\sigma_1,\ldots,\sigma_r:X\to \mathbb R^n\) be arbitrary maps. Define
\(Z_i=\langle \sigma_i(U),\sigma_i(U')\rangle\). Then there are
\(i\in[r]\) and \(\lambda\ge -1\) such that
\begin{equation} \label{eq:ramsey_lemma}
    \mathbb P\!\left[
        Z_i\ge \lambda
        \ \text{and}\
        Z_j\ge -1 \text{ for every } j\neq i
    \right]
    \ge
    \beta_r \exp\!\left(-C_r\sqrt{\lambda+1}\right),
\end{equation}
where \(C_r\le C r\log(r)\) and
\(\beta_r\ge \exp(-C r\log(r))\) for an absolute constant \(C>0\).
\end{lemma}

\begin{remark}
    The primary improvement in our lemma is better control on $C_r$. In particular, Lemma~3.1 of \cite{balister2024upperboundsmulticolorramsey} gives $C_r = \Theta(r^{3/2})$. We discuss why the robust OR polynomial leads to this improvement in Subsection \ref{subsec:rob_or}.
\end{remark}

Inserting this improved geometric input into the book algorithm of \cite{balister2024upperboundsmulticolorramsey} gives our second main theorem.

\begin{theorem}
\label{thm:improved-multicolor-ramsey}
There are absolute constants \(c,C_0>0\) such that, for every \(r\ge 2\)
and every \(k\ge C_0r^{14}(\log(2r))^{12}\),
\[
    R_r(k)
    \le
    \exp\!\left(
        -\frac{c k}{r^9(\log(2r))^6}
    \right)r^{rk}.
\]
\end{theorem}

In the following subsection, we will describe the robust OR polynomial and outline how we apply it towards proving our main results.

\subsection{The Robust OR Polynomial and Outlines of our Proofs} \label{subsec:rob_or}

We now define the robust OR polynomial used throughout this paper. Suppose that we have \(r\) real inputs
\(
    t_1,\ldots,t_r \in D,
\)
where \(D\subseteq \mathbb R\) is a common admissible domain. Let
\(\calI \subseteq D\) denote a small set of ``bad'' values. The task for the robust OR polynomial is to detect the OR event that there exists an index \(i\in [r]\) such that \(t_i\in \calI\). To be more specific, the robust OR polynomial $Q_r:D^r\to\R$ has the following property:
\[\exists i, t_i\in \calI\ \  \Rightarrow \ \ Q_r(t_1,\ldots,t_r) < 0.\]
Moreover, for our applications in this paper, we wish to test the OR event for positive semidefinite inputs. Therefore, to preserve the positive semidefiniteness, we have an additional requirement for $Q_r$:
\[
\text{All coefficients of } Q_r \text{ are nonnegative.}
\]

As a motivating example, we first consider this task on a hypercube \(\calH^r = \{\pm 1\}^r\), i.e., we want to detect if any of the inputs is $-1$. For $t = (t_1,\ldots, t_r)$, define the polynomial $\orhypercube:\calH^r \to \R$:
\begin{equation}\label{def:fail-or}
    \orhypercube(t) :=     \sum_{s=1}^r e_s(t_1,\ldots,t_r) = \prod_{i = 1}^r (1+t_i) - 1,
\end{equation}
where $e_s(t_1,\ldots,t_r) = \sum_{\substack{S\subseteq [r]\\ |S|=s}} \prod_{i\in S} t_i$  is the \(s\)-th elementary symmetric polynomial. This function has all coefficients nonnegative, and detects the OR event: if any $t_i$ has value $-1$, $f$ takes value $-1$.

Though this construction works on $\calH^r$, it lacks the robustness required when the inputs belong to a larger region. Suppose the inputs $t_1,\ldots,t_r$ lie in $[-1, \infty)$ and we want to detect if any of the inputs is in $[-1, -1+\eta]$, where $\eta > 0$ represents a small slack around $-1$. The function $f$ defined in~\eqref{def:fail-or} fails because taking $t_1 = -1 + \eta$ and $t_2 = \cdots = t_r = M$ with a large $M$ makes $f(t) = \eta (1+M)^{r-1} - 1$, which can be arbitrarily large. This observation motivates the following definition of a robust OR polynomial:
\begin{equation}\label{def:robust-or}
    Q_{r}(t_1,\ldots,t_r)
    :=
    \sum_{s=1}^r (1-(1-\frac{1}{r})^s)e_s(t_1,\ldots,t_r) = \prod_{i=1}^r(1+t_i) - \prod_{i=1}^r(1+(1-\frac{1}{r})t_i).
\end{equation}

All the coefficients of this polynomial are nonnegative.
In the next lemma, we show that $Q_r$ can detect the OR event described above. We defer its proof to the appendix.

\begin{lemma} \label{lem:stren-or}
Let \(r\ge2\), \(0<\eta\le(8er)^{-1}\), and $Q_r$ be defined as in~\eqref{def:robust-or}. Suppose \(t_i\ge-1\) for every \(i\in [r]\).  If there exists at least one coordinate in the interval \([-1, -1+\eta]\), then
\[
Q_r(t_1,\ldots,t_r)\le -\frac12 r^{-r}.
\]
\end{lemma}

Although Lemma~\ref{lem:stren-or} may seem restrictive because it works specifically for detecting coordinates in $[-1,-1+\eta]$, one can build versatile polynomials by composing this robust OR polynomial with certain compressors, which regularize the inputs to be in this specific region. Indeed, such compositions are key towards our proofs of Theorems \ref{thm:schrijver-quasi-tensorization} and \ref{thm:improved-multicolor-ramsey}.

Additionally, we note that another important feature of $Q_r$ is that all the coefficients are nonnegative. By Schur's product theorem, we are able to compose this OR polynomial with other positive semidefinite kernels while preserving the positive semidefiniteness. Here, a kernel $K$ on a set $X$ is defined as a function \(K:X\times X\to \mathbb R\). We say that a kernel is positive semidefinite if, for every finite choice \(x_1,\ldots,x_m\in X\), the matrix \((K(x_i,x_j))_{i,j=1}^m\) is positive semidefinite.

We now outline the proofs of Theorems \ref{thm:schrijver-quasi-tensorization} and \ref{thm:improved-multicolor-ramsey}, and discuss how they apply the robust OR polynomial.

\paragraph{Schrijver Quasi-Tensorization. }

To prove Theorem~\ref{thm:schrijver-quasi-tensorization}, we leverage duality to show that for every graph $G = (V,E)$, if one can produce a positive semidefinite kernel $K: V\times V \to \R$ satisfying
\begin{equation}\label{eq:kernel-requirement-for-schrijver-bound}
    K(v,v) \le \delta,\ \forall v \in V, \qquad K-\gamma J \succeq 0,\qquad K(u,v) \le \beta\ 
\text{for } u \ne v \text{ and } uv \notin E
\end{equation}
for some \(0 \le \beta < \gamma\) and \(\delta\), then $\vartheta'(G)\leq \frac{\delta-\beta}{\gamma-\beta}.$ Thus our proof aims to build a positive semidefinite kernel $K$ satisfying~\eqref{eq:kernel-requirement-for-schrijver-bound} on the strong product graph $G_1\boxtimes\cdots\boxtimes G_r$ to get the bound on $\vartheta'(G_1\boxtimes\cdots\boxtimes G_r)$.

The proof starts with one-coordinate kernels $K_1,\ldots,K_r$ from the Schrijver SDPs on $G_1,\ldots,G_r$, which by definition are nonpositive on nonedges. We shift each kernel $K_i$ to make it strictly negative over nonedges. 
Then, for each \(i\), we pass the kernel \(K_i\) through a sign compressor \(T_i\). The compressor \(T_i\) is a rational function, with nonnegative coefficients in both its numerator and denominator, that behaves like a sign function on the relevant bounded interval. Concretely, \(T_i\) is constructed by approximating $\operatorname{sgn}(x)=\frac{x}{\sqrt{x^2}}$ via sinc-quadrature rational approximation. Thus, on the interval of interest, \(T_i\) sends negative inputs to values close to \(-1\). We then obtain the desired kernel \(K\) by first sending the one-coordinate kernels to the compressors, then composing with the robust OR polynomial defined in~\eqref{def:robust-or}, and lastly clearing the denominators.

\paragraph{Improved \(r\)-dependence for multicolor Ramsey.}
Before explaining our improvement, we briefly recall the book algorithm of
\cite{balister2024upperboundsmulticolorramsey}. In an \(r\)-edge-colored complete graph, a
color-\(i\) \((t,m)\)-book is a pair of disjoint vertex sets \((S,W)\) such that \(|S|=t\),
\(|W|=m\), \(S\) spans a color-\(i\) clique, and every edge between \(S\) and \(W\) has color
\(i\). We call \(S\) the spine and \(W\) the page set. Thus, if \(W\) contains a color-\(i\)
clique of size \(k-t\), then this clique together with \(S\) forms a monochromatic \(K_k\).
At a high level, the book algorithm repeatedly tries to grow a large monochromatic book. At each
stage, it either adds a vertex to the spine of some color, or it finds a density boost in one
color that can be exploited later. The geometric lemma is the local input which guarantees that
one of these two useful outcomes occurs on a sufficiently large surviving subset. 

In the present
paper, we do not modify the book algorithm itself. Instead, we improve the geometric lemma that is
fed into the same algorithm, and this improved input gives better bounds for \(R_r(k)\). Given any finite set $X$, the geometric lemma concerns arbitrary maps \(\sigma_1,\ldots,\sigma_r:X\to \mathbb R^n\). More specifically, the
geometric lemma guarantees that there are \(i\in[r]\) and \(\lambda\ge -1\) such that
\[
    \mathbb P\!\left[
        \langle \sigma_i(U),\sigma_i(U')\rangle\ge \lambda
        \ \text{and}\
        \langle \sigma_j(U),\sigma_j(U')\rangle\ge -1
        \text{ for every }j\neq i
    \right]
    \ge
    \beta_r\exp\!\left(-C_r\sqrt{\lambda+1}\right).
\]
In the geometric lemma of \cite{balister2024upperboundsmulticolorramsey}, $C_r = \Theta(r^{3/2})$ whereas our improvement gives \(C_r\le C r\log(r)\) and
\(\beta_r\ge \exp(-C r\log(r))\).
In the book algorithm, each coordinate corresponds to one color and
a large inner product in one coordinate corresponds to a density boost in that color, while lower
bounds in the remaining coordinates ensure that the other color densities do not collapse. 

We now explain why the proof of such a lemma reduces to constructing a suitable entire function.
The moment-positivity input (Lemma 3.2 of \cite{balister2024upperboundsmulticolorramsey}) says that if \(Z_i=\langle \sigma_i(U),\sigma_i(U')\rangle\), then
\(\mathbb E\prod_{i=1}^r Z_i^{a_i}\ge 0\) for every
\(a_1,\ldots,a_r\in\mathbb Z_{\ge 0}\). Thus, any entire function
\(P(Z_1,\ldots,Z_r)\) with nonnegative Taylor coefficients has nonnegative expectation. The goal
is therefore to choose \(P\) so that it is negative when some coordinate is below the threshold
\(-1\), but has controlled growth when all coordinates are at least \(-1\). The non-negativity of
\(\mathbb E P(Z_1,\ldots,Z_r)\) then forces the good event to have enough mass, so we can extract the threshold \(\lambda\).

Our improvement is in the construction of this test function. We first build a Bessel-function
compressor $T$ that has nonnegative Taylor coefficients, maps every \(x\le -1\) close to \(-1\), and satisfies the growth bound \(1+T(x)\le \exp\{O(\log(2r))\sqrt{x+1}\}\) for all \(x\ge -1\). We then compose these compressed coordinates with the robust OR polynomial so that if all \(x_i\ge -1\), then the growth of \(P\) is controlled by the
growth of the compressor, giving
\[
    P(x_1,\ldots,x_r)
    \le
    2 r^r \cdot 
    \exp \big(O(r\log(r))\sqrt{\max_i x_i+1} \big).
\]
This is the point where the robust OR polynomial improves the dependence on \(r\): it separates
the OR detection step from the growth control step, giving \(C_r=O(r\log r)\) and
\(\beta_r\ge \exp(-O(r\log r))\).
Finally, inserting this improved geometric input into the unchanged book algorithm of
\cite{balister2024upperboundsmulticolorramsey} yields
the improved Ramsey bound $R_r(k) \le \exp\left\{
        -\Omega\left(\frac{k}{r^9(\log(2r))^6}\right)
    \right\}r^{rk}$ for \(k\ge C r^{14}(\log(2r))^{12}\).

\subsection{Organization of the Paper}

 Section \ref{sec:sch-tens} contains our proof of Theorem \ref{thm:schrijver-quasi-tensorization} and its application to $r$-way acute-free families (Corollary \ref{cor:Mr(n)}). Section \ref{sec:ramsey} contains our proofs of Lemma \ref{lem:or-composed-geometric-lemma} and Theorem \ref{thm:improved-multicolor-ramsey}.

\paragraph{Notation} Let $\symm$ be the set of $n\times n$ real symmetric matrices. A matrix $A\in \symm$ is said to be positive semidefinite (denoted by $A\succeq 0$) if all its eigenvalues are nonnegative. We denote the identity matrix by $I$ and the all-ones matrix by $J$.
For a matrix $A$, we denote the matrix keeping only the diagonal entries of $A$ and setting other entries to $0$ by $\text{diag}(A)$. For a vector $x\in \R^{n}$, we denote its $i$-th entry by $x_i$. For two vectors $x,y\in\R^n$, we denote the standard inner product between them by $\ip{x}{y}$. For two matrices $A,B\in\R^{n\times n}$, we denote the Kronecker product between them by $A\otimes B$, and the Schur product between them by $A\circ B$. We define $A^{\circ k}:= \underbrace{A \circ A \circ \cdots \circ A}_{k\text{ times}}$ to be $k$-th Schur product power of a matrix $A$ and $A^{\circ 0} = J$. For $k$ matrices $A_1,\ldots,A_k\in \R^{n\times n}$, we denote $A_1\circ A_2\circ \cdots \circ A_k$ by $\bigcirc_{i=1}^k A_i$.  We denote the hypercube $\{\pm 1\}^n$ by $\Hyp$ and the unit sphere by $S^{n-1}$.  We denote the set of polynomials with nonnegative coefficients to be $\mathbb{R}_{\ge 0}[x]$, the set of nonnegative integers to be $\mathbb Z_{\geq 0}$. We denote the total variation of a function $f$ by $\operatorname{Var}(f)$. For $\alpha = (\alpha_1,\ldots,\alpha_r)\in \mathbb Z_{\geq 0}^r$ and $x = (x_1,\ldots,x_r)$, we define $x^\alpha$ to be $\prod_{i=1}^rx_i^{\alpha_i}$. We identify a symmetric kernel \(K:X\times X\to\mathbb R\) with its
matrix \((K(x,y))_{x,y\in X}\) whenever \(X\) is finite. 

\section{Schrijver Number Quasi-tensorization} \label{sec:sch-tens}

In this section, we prove Theorem \ref{thm:schrijver-quasi-tensorization}. To upper bound the Schrijver number, we aim to find a good feasible solution to the dual of Schrijver SDP~\eqref{def:schrijver-sdp-primal}. Let $G = (V,E)$ be a graph, the dual of  Schrijver SDP~\eqref{def:schrijver-sdp-primal} on $G$ is as follows:
\begin{equation}\label{eq:schrijver-dual-2}
\begin{aligned}
    \min_{\lambda,A}\quad & \lambda\\
\text{s.t.}\quad
& \lambda I-A\succeq0 \\
& A_{vv}=1
\qquad \forall v\in V\\
& A_{uv}\ge1
\qquad \forall u\ne v,\ uv\notin E. 
\end{aligned}
\end{equation}
From~\cite{Schrijver1979}, strong duality holds for the Schrijver SDP. Before we delve into the proof, 
we first isolate a general dual-certificate lemma for the Schrijver number in terms of positive semidefinite kernels, which will be useful for the proof of Theorem~\ref{thm:schrijver-quasi-tensorization}.

\begin{lemma}\label{lem:general-schrijver-bound}
    Let \(G=(V,E)\) be a graph and \(K : V \times V \to \mathbb R\) be a kernel.
If there are real numbers \(0 \le \beta < \gamma\) and \(\delta\) such that
\(
K(v,v) \le \delta \ \text{for every } v \in V
\), $K-\gamma J \succeq 0$,
and
\(
K(u,v) \le \beta\ 
\text{whenever } u \ne v \text{ and } uv \notin E.
\)
Then, we have
\[
\vartheta'(G) \le \frac{\delta-\beta}{\gamma-\beta}.
\]
\end{lemma}
\begin{proof}
    We aim to find a dual feasible solution to SDP~\eqref{eq:schrijver-dual-2}. Define a matrix $A\in \R^{V\times V}$ as follows:
\[
    A_{uv} = \begin{cases} 
1, & u = v \\ 
\frac{\gamma - K(u,v)}{\gamma - \beta}, & u\neq v 
\end{cases}.
\]
Let $\lambda = \frac{\delta - \beta}{\gamma - \beta}$.
We claim that \((\lambda, A)\) is feasible to the Schrijver SDP~\eqref{eq:schrijver-dual-2}. If $u \neq v$ and $uv\notin E$, we have 
\[
    A_{uv} = \frac{\gamma - K(u,v)}{\gamma - \beta}\geq \frac{\gamma - \beta}{\gamma - \beta} = 1. 
\]
By a simple calculation, we have that
\[
    \lambda I - A =\frac{1}{\gamma - \beta} (K - \gamma J) + \frac{1}{\gamma - \beta}(\delta I  - \text{diag}(K)). 
\]
The first matrix $K - \gamma J$ is positive semidefinite by assumption. The second matrix is positive semidefinite since $\delta\geq K(v,v)$ for all $v\in V$. Thus, $\lambda I - A\succeq 0$. Therefore, \((\lambda, A)\) is feasible to SDP~\eqref{eq:schrijver-dual-2} and we have
\[
\vartheta'(G)
\le
\lambda
=
\frac{\delta-\beta}{\gamma-\beta}.
\]
\end{proof}

At a high level, the proof of Theorem~\ref{thm:schrijver-quasi-tensorization} applies the OR polynomial in~\eqref{def:robust-or} to a positive semidefinite kernel from the Schrijver SDP on each graph. To do that, we define a compressor to regularize the input Schrijver kernels to the working range of the OR polynomial. Ideally, we want to cluster the negative inputs to the interval close to $-1$, capturing the behavior of a sign function.

As $\text{sgn}(x) = x\frac{1}{\sqrt{x^2}}$, we utilize the following lemma, which gives a rational approximation to $\frac{1}{\sqrt{x}}$ where both the numerator and denominator have low degree and nonnegative coefficients. It is related to sinc quadrature and a similar lemma also appears in, for example~\cite[Lemma 9]{liu2026smoothed}. We defer the proof of it to subsection \ref{subsec:sqrt_approx}.

\begin{lemma}
\label{lem:approx-sqrt-x}

There is an absolute constant \(c_0>0\) such that, for every
\(M\ge 2\) and \(0<\tau\le 1\), there are polynomials
\(F,G\in\mathbb{R}_{\ge 0}[x]\) satisfying
\[
\deg F,\deg G
\le
c_0\log\left(\frac{M}{\tau}\right),
\qquad
G(x)>0
\quad
(x\ge 0),
\]
and
\[
\frac{7}{8}
\le
\sqrt{x}\frac{F(x)}{G(x)}
\le
\frac{9}{8}
\]
for every \(x\in[\tau^2,{M^2}]\).

\end{lemma}

With lemma~\ref{lem:approx-sqrt-x} in hand, we will proceed to the proof of Schrijver quasi-tensorization.

\begin{proof}[Proof of Theorem \ref{thm:schrijver-quasi-tensorization}]
    For each $i\in [r]$, pick $(\lambda_i, A_i)$ to be an optimal solution of the dual Schrijver SDP~\eqref{eq:schrijver-dual-2} on $G_i$. Let $u,v\in V_i$, we define a kernel $K_i:V_i\times V_i\to \R$ as
    \[
    K_i(u,v) = \begin{cases}
        \lambda_i, & u = v \\ 1 - (A_i)_{uv}, & u\neq v 
    \end{cases}.
    \]
Equivalently, $K_i = \lambda_i I - A_i + J.$ Since $(\lambda_i, A_i)$ is feasible to SDP~\eqref{eq:schrijver-dual-2}, we have $K_i - J = \lambda_i I - A_i \succeq 0$.

Denote $G_1\boxtimes \ldots \boxtimes G_r$ by $G$ and $V_1\times \ldots \times V_r$ by $V$.
For $x = (x_1, \ldots, x_r), y = (y_1,\ldots,y_r) \in V$ and $i\in [r]$, 
we define a kernel $\productkernel{i}: V \times V\to\R{}$ as
\( \productkernel{i}(x,y) = K_i(x_i,y_i)
\). Since $K_i - J\succeq 0$, we have
\[
    \productkernel{i} - J = J_{V_1} \otimes \cdots\otimes J_{V_{i-1}} \otimes (K_{i} - J_{V_i}) \otimes J_{V_{i+1}} \otimes \cdots \otimes J_{V_r} \succeq 0.
\] 
We also know that for any $x\in V$, $\productkernel{i}(x,x) = K_i(x_i,x_i) = \lambda_i$. For $x = (x_1, \ldots, x_r), y = (y_1,\ldots,y_r) \in V$ and $i\in [r]$, we now define a translated kernel $Z_i$ as
$Z_i(x,y) = 2\productkernel{i}(x,y) - 1$, which can equivalently be expressed in the matrix form as $Z_i = 2\productkernel{i} - J$. 

From Lemma~\ref{lem:approx-sqrt-x}, there are polynomials
\(F_i,H_i\in\mathbb{R}_{\ge 0}[x]\) satisfying
\(
\deg F_i,\deg H_i
\le
c_0\log\left(2\lambda_i\right)\) for some absolute constant \(c_0 > 0 \), \(
H_i(x)>0 \) for \(x\ge 0\)
and
\[
\frac{7}{8}
\le
\sqrt{x}\frac{F_i(x)}{H_i(x)}
\le
\frac{9}{8}
\]
for every \(x\in[1,4\lambda_i^2]\). Let $L$ be the smallest even integer such that \((\frac{1}{8})^L\le\frac1{8er}
\). Define
\[
T_i(x) = \left(1 + x\frac{F_i(x^2)}{H_i(x^2)}\right)^L - 1.
\]
Since $L$ is even, we have $T_i(x)\geq -1$ for all $x$. Let $Q_r$ be defined as in Lemma~\ref{lem:stren-or} and set 
\[
\mathcal{P}_r(z_1,\ldots,z_r) = \left(\prod_{i=1}^r H_i(z_i^2)^L\right) Q_{r}(T_1(z_1),\ldots,T_r(z_r)).
\] For $x = (x_1, \ldots, x_r), y = (y_1,\ldots,y_r) \in V$, we now define a kernel $K:V\times V\to \R$
as \[
K(x,y) = \mathcal{P}_r(Z_1(x,y),\ldots,Z_r(x,y)).
\]
We next show that the kernel $K$ satisfies the conditions in Lemma~\ref{lem:general-schrijver-bound} for some appropriately chosen $\gamma, \beta,\delta$.
Note that
\begin{equation}\label{eq:expansions}
    \begin{aligned} 
    \mathcal{P}_r(z_1,\ldots,z_r) &= \left(\prod_{i=1}^r H_i(z_i^2)^L\right) Q_{r}(T_1(z_1),\ldots,T_r(z_r))\\
        &= \sum_{\emptyset\neq S\subseteq[r]}
        (1-(1-\frac{1}{r})^{|S|})
        \prod_{i\in S} \left((H_i(z_i^2) + z_i F_i(z_i^2))^L - H_i(z_i^2)^L\right)
        \prod_{i\notin S} H_i(z_i^2)^L.
\end{aligned}
\end{equation}
Since polynomials $H_i,F_i$ have nonnegative coefficients,
$\mathcal P_r$ has the following expansion
\[
        \mathcal P_r(z_1,\ldots,z_r)
        =
        \sum_{\alpha\in \mathbb{Z}_{\geq 0}^r} c_\alpha
        z_1^{\alpha_1}\cdots z_r^{\alpha_r},
        \qquad
        c_\alpha\ge0.
\]
Therefore, we have
\begin{align*}
    K &= \sum_\alpha c_\alpha Z_1^{\circ\alpha_1}\circ \cdots \circ Z_r^{\circ \alpha_r} \\
    &= \sum_\alpha c_\alpha \left( J +  2(\productkernel{1} - J) \right)^{\circ\alpha_1} \circ \cdots \circ \left( J +  2(\productkernel{r} - J) \right)^{\circ \alpha_r}\\
    &= \sum_\alpha c_\alpha \bigcirc_{i = 1}^r \left( J + \sum_{j = 1}^{\alpha_i} \binom{\alpha_i}{j} \left(  2(\productkernel{i} - J) \right)^{\circ j}\right).
\end{align*}
Since $\left( J +  2(\productkernel{i} - J) \right)^{\circ\alpha_i} = \left( J + \sum_{j = 1}^{\alpha_i} \binom{\alpha_i}{j} \left(  2(\productkernel{i} - J) \right)^{\circ j}\right)$ and $2(\productkernel{i} - J) \succeq 0$, by letting $\gamma = \sum_\alpha c_\alpha$, it follows that
\begin{equation}\label{eq:gamma-condition}
    K - \gamma J \succeq 0.
\end{equation}

Pick any $x = (x_1, \ldots, x_r), y = (y_1,\ldots,y_r)\in V$ such that $x\neq y, xy\notin E$. By definition of strong product, there is a coordinate $j$ such that $x_j \neq y_j$ and $x_jy_j\notin E_j$. Then $Z_j(x,y) = 2\productkernel{j}(x,y) - 1 = 2K_j(x_j,y_j) - 1 \leq -1$. Also since $Z_j-J = 2(\productkernel{j} - J)\succeq 0$, we have 
\(|Z_j(x,y)| \leq Z_j(x,x) = 2\lambda_j - 1\). Therefore, $Z_j(x,y)^2 \in [1,4\lambda_j^2]$. By Lemma~\ref{lem:approx-sqrt-x} and the definition of $T$, we have $T_j(Z_j(x,y)) \in [-1, -1 + (\frac{1}{8})^L]$ and 
\[
Q_r(T_1(Z_1(x,y)),\ldots,T_r(Z_r(x,y))) \leq 0.
\]
Since \(\prod_{i=1}^r H_i(z_i^2)^L\) is always strictly positive, we have
\begin{equation}\label{eq:beta-condition}
    K(x,y)\leq 0.
\end{equation}
Finally we bound the diagonal term of the kernel $K$. Let $d_i$ denote $\max (\deg(F_i), \deg(H_i))
$ and $D_i$ denote the degree of $\mathcal{P}_r$ in variable $z_i$, then by the expansion in~\eqref{eq:expansions}, we have
\[
D_i \leq L(2d_i + 1) \leq C' \log (r)\log(2\lambda_i)
\]
for some absolute constant $C' >0$. Thus, we have
\begin{equation}\label{eq:delta-condition}
    K(x,x) = \mathcal P_r(Z_1(x,x),\ldots, Z_r(x,x)) = \sum_\alpha c_\alpha Z_1(x,x)^{\alpha_1}\cdots Z_r(x,x)^{\alpha_r} \leq (\sum_\alpha c_\alpha ) \prod_{i=1}^r (2\lambda_i)^{D_i}
\end{equation}
Then, apply Lemma~\ref{lem:general-schrijver-bound} by picking $\gamma = \sum_\alpha c_\alpha, \beta = 0, \delta = (\sum_\alpha c_\alpha ) \prod_{i=1}^r (2\lambda_i)^{D_i}$, we have
\[
\vartheta'(G)\leq \frac{(\sum_\alpha c_\alpha ) \prod_{i=1}^r (2\lambda_i)^{D_i}}{\sum_\alpha c_\alpha} = \prod_{i=1}^r (2\lambda_i)^{D_i}\leq \prod_{i=1}^r \vartheta'(G_i)^{C\log r \log \vartheta'(G_i)}
\]
for some absolute constant $C > 0$.

\end{proof}

\begin{remark}\label{rem:quasipolynomial-dependency}
    The quasipolynomial dependency on $\vartheta'(G_i)$ in Theorem~\ref{thm:schrijver-quasi-tensorization} comes from the fact that the polynomials $F_i,H_i\in \R_{\geq 0}[x]$ in the compressor have degree $O(\log \vartheta'(G_i))$.
    The $\log r$ dependency in the exponent comes from the choice of the power $L$ in the compressor. 
\end{remark}

\subsection{Proof of Corollary \ref{cor:Mr(n)}}

In this section, we present the proof of Corollary \ref{cor:Mr(n)}. The strategy is a straightforward application of Theorem \ref{thm:schrijver-quasi-tensorization}. We first bound the Schrijver number in one dimension, and then leverage the quasi-tensorization Schrijver bound to prove a quasipolynomial bound for general $r$-way acute-free families.

The one-coordinate certificate is standard in spherical-code theory; see, for example,~\cite[Example 4.5]{DelsarteGoethalsSeidel1977}. We include it here for completeness and also to demonstrate the general connection between positive semidefinite kernels and Schrijver bounds.
We first define a function $q(t) = n(t + t^2)$ and a kernel $K:\sphere \times \sphere \to \R$ as follows:
\begin{equation}\label{eq:sphere-kernel}
    K(x,y) = q(\ip{x}{y}).
\end{equation}
Let $\calF = \{\x{1},\ldots,\x{k}\}\subseteq \sphere$ be any finite set of points on the unit sphere. We define a graph $G_\calF$, whose vertex set is $\calF$ and $u,v\in \calF$ is connected if $\ip{u}{v} > 0$. The next lemma shows that the kernel in~\eqref{eq:sphere-kernel} gives $\vartheta'(G_\calF)\leq 2n$. 
\begin{lemma}\label{lem:sphere-one-coordinate}
    Let $\calF = \{\x{1},\ldots,\x{k}\}\subseteq \sphere$ be any finite set of points on the unit sphere and $G_\calF$ be a graph on $\calF$ such that $u,v\in \calF$ is connected if $\ip{u}{v} > 0$.
    Then, we have
    $$\vartheta'(G_\calF)\leq 2n.$$
\end{lemma}

\begin{proof}
    Let $K_\calF$ be the matrix $\{K_{\x{i},\x{j}}\}_{i,j = 1}^k$. We try to show that $K_\calF$ satisfies the condition in Lemma~\ref{lem:general-schrijver-bound} for some $\gamma = 1, \delta = 2n, \beta = 0$. We first show that $K_\calF - J\succeq 0$. This is equivalent to showing
    \[
        a^\top (K_\calF - J) a = n\sum_{i,j} a_ia_j\ip{\x{i}}{\x{j}} + n\sum_{i,j} a_ia_j\left(\ip{\x{i}}{\x{j}}\right)^2 - \left(\sum_i a_i\right)^2\geq 0
    \]
    for all $a = (a_1,\ldots, a_k)\in\R^k$. The first term is nonnegative because $$\sum_{i,j} a_ia_j\ip{\x{i}}{\x{j}} = \norm{\sum_i a_i\x{i}}_2^2\geq 0.$$
    Let $B = \sum_i a_i \x{i}x^{(i)^\top}$, then we have $\sum_{i,j} a_ia_j\left(\ip{\x{i}}{\x{j}}\right)^2 = \norm{B}_F^2$. Since $\norm{\x{i}}_2 = 1$, we have $\Tr(B) = \sum_i a_i.$ Thus, by Cauchy-Schwarz inequality, we have
    \[
        \left(\sum_i a_i\right)^2 = \Tr(B)^2 \leq \|B\|_F^2\|I\|_F^2 = n\norm{B}_F^2 = n\sum_{i,j} a_ia_j\left(\ip{\x{i}}{\x{j}}\right)^2. 
    \]
    Thus, $a^\top (K_\calF - J) a \geq 0$ for all $a\in\R^{k}$ and $K_\calF - J \succeq 0$. 

    For all $i\neq j\in[k]$ satisfying $\x{i}\x{j}\notin E(G_\calF)$, we have $\ip{\x{i}}{\x{j}} \leq 0$. Therefore $K(\x{i},\x{j}) = q(\ip{\x{i}}{\x{j}}) \leq 0$. Also, note that for all $i\in [k]$, we have $K(\x{i},\x{i}) = q(1) = 2n$. Applying Lemma~\ref{lem:general-schrijver-bound} with $\gamma = 1,\delta = 2n,\beta = 0$ gives
    \(
    \vartheta'(G_\calF) \leq 2n.
    \)
\end{proof}

Now we use Theorem~\ref{thm:schrijver-quasi-tensorization} to prove Corollary~\ref{cor:Mr(n)}.

\begin{proof}[Proof of Corollary~\ref{cor:Mr(n)}]
    Pick any largest $r$-way acute-free set $\calD = \{\mathbf {x}^{(1)},\ldots,\mathbf{x}^{(k)} \}\subseteq (\sphere)^r$. For all $i\in [r]$ we define $\calF_i$ to be the set which each element is the $i$-th coordinate of each element in $\calD$ and $G_{\calF_i}$ to be a graph on $\calF_i$ such that $u,v\in \calF_i$ is connected if $\ip{u}{v} > 0$. We claim that, for all $i\in [r]$, $G_{\calF_i}$ is not a  complete graph and thus $\vartheta'(G_{\calF_i}) \geq 2$. Suppose, for contradiction,
that $G_{\calF_i}$ is complete for some $i$. Then the $i$-th coordinate of the vectors in $\calD$ will never
witness the non-acuteness condition: for any two distinct points of $\calD$,
the $i$-th coordinate inner product is positive. Now form
\[
\calD'=\{(x_1,\ldots,-x_i,\ldots,x_r):(x_1,\ldots,x_i,\ldots,x_r)\in\calD\}.
\]
The union $\calD\cup\calD'$ is still
$r$-way acute-free, which contradicts the maximality of $\calD$.
  Then, by Lemma~\ref{lem:sphere-one-coordinate}, for every $i\in [r]$ we have
$$\vartheta'(G_{\calF_i})\leq 2n.$$
    Applying Theorem~\ref{thm:schrijver-quasi-tensorization} gives
    $$
    \vartheta'(G_{\calF_1} \boxtimes \cdots \boxtimes G_{\calF_r}) \leq (2n)^{C_0r\log r\log(2n)}  $$
    for some absolute constant $C_0 > 0$.
    Then, by definition of $r$-way acute-free set, $\calD$ is an independent set in $G_{\calF_1} \boxtimes \cdots \boxtimes G_{\calF_r}$. Thus, we have $M_r(n) = |\calD|\leq \vartheta'(G_{\calF_1} \boxtimes \cdots \boxtimes G_{\calF_r}) \leq (2n)^{C_0r\log r\log(2n)}$. The exact same proof holds for $M_r^{\pm}(n)$.
\end{proof}

\subsection{Proof of Lemma \ref{lem:approx-sqrt-x}} \label{subsec:sqrt_approx}

\begin{proof}

Fix \(\varepsilon=\frac18\) and choose constants \(h>0\) and
\(C_0>0\), depending only on \(\varepsilon\), such that

\begin{equation}
\label{eq:epsilon_choice_variable}
\frac{2h}{\pi}
+
\frac{4he^{-C_0}}{\pi(1-e^{-h})}
\le
\varepsilon.
\end{equation}
Additionally, set 
\[
m_0
=
\left\lceil
\frac{\log(M/\tau)+C_0}{h}
\right\rceil
\]
and define the polynomials \(F,G\in\mathbb{R}_{\ge 0}[x]\) as
follows:
\[
F(x)
=
\frac{2h}{\pi}
\sum_{k=-m_0}^{m_0}
e^{kh}
\prod_{\substack{\ell=-m_0\\ \ell\ne k}}^{m_0}
(x+e^{2\ell h}),
\]
\[
G(x)
=
\prod_{k=-m_0}^{m_0}(x+e^{2kh}).
\]
Note that $\deg F,\deg G
=
O\left(\log\left(\frac{M}{\tau}\right)\right).$
Since \(h\) and \(C_0\) are absolute constants, there is an
absolute constant \(c_0>0\) such that
\[
\deg F,\deg G
\le
c_0\log\left(\frac{M}{\tau}\right).
\]
Moreover, every factor \(x+e^{2kh}\) is strictly positive for
\(x\ge 0\), and therefore $G(x)>0$ whenever $x\ge 0$. We claim that the polynomials \(F\) and \(G\) provide the desired
approximation. By construction,
\[
\frac{F(x)}{G(x)}
=
\frac{2h}{\pi}
\sum_{k=-m_0}^{m_0}
\frac{e^{kh}}{x+e^{2kh}}.
\]
Perform the change of variables \(x=e^{2y}\). Since
\(x\in[\tau^2,M^2]\), we have $
y\in[\log\tau,\log M]$
and therefore $|y|
\le
\log\left(\frac{M}{\tau}\right).$
Thus,
\begin{align}
\label{eq:simp_expansion_variable}
\sqrt{x}\,\frac{F(x)}{G(x)}
&=
e^y\cdot
\frac{2h}{\pi}
\sum_{k=-m_0}^{m_0}
\frac{e^{kh}}{x+e^{2kh}}
\nonumber\\
&=
\frac{2h}{\pi}
\sum_{k=-m_0}^{m_0}
\frac{e^y e^{kh}}{e^{2y}+e^{2kh}}
\nonumber\\
&=
\frac{2h}{\pi}
\sum_{k=-m_0}^{m_0}
\frac{1}{e^{y-kh}+e^{kh-y}}
\nonumber\\
&=
\frac{h}{\pi}
\sum_{k=-m_0}^{m_0}
\operatorname{sech}(kh-y).
\end{align}
Now, observe that
\begin{align}
\label{eq:tv_approx_variable}
\left|
h\sum_{k\in\mathbb Z}\operatorname{sech}(kh-y)-\pi
\right|
&=
\left|
h\sum_{k\in\mathbb Z}\operatorname{sech}(kh-y)
-
\int_{\mathbb R}\operatorname{sech}(u-y)\,du
\right|
\nonumber\\
&\le
h\,\operatorname{Var}\!\left(
u\mapsto\operatorname{sech}(u-y)
\right)
\nonumber\\
&=
2h,
\end{align}
where the inequality follows from the Riemann-sum variation bound
\[
\left|
h\sum_{k\in\mathbb Z}f(kh)
-
\int_{\mathbb R}f(u)\,du
\right|
\le
h\,\operatorname{Var}(f),
\]
and the final equality follows because $
\operatorname{Var}\!\left(
u\mapsto\operatorname{sech}(u-y)
\right)
=
2$. Next we control the truncation error.

\begin{align}
\label{eq:truncation_variable}
h\sum_{|k|>m_0}\operatorname{sech}(kh-y)
&\le
2h\sum_{|k|>m_0}e^{-|kh-y|}
\qquad
(\operatorname{sech}(u)\le 2e^{-|u|})
\nonumber\\
&\le
2h\sum_{|k|>m_0}e^{-|k|h+|y|}
\qquad
(|kh-y|\ge |k|h-|y|)
\nonumber\\
&\le
2h e^{\log(M/\tau)}
\sum_{|k|>m_0}e^{-|k|h}
\qquad
\left(
|y|\le\log\left(\frac{M}{\tau}\right)
\right)
\nonumber\\
&=
4h e^{\log(M/\tau)}
\sum_{k=m_0+1}^{\infty}e^{-kh}
\nonumber\\
&=
\frac{
4h e^{\log(M/\tau)}e^{-(m_0+1)h}
}{
1-e^{-h}
}
\nonumber\\
&\le
\frac{4he^{-C_0}}{1-e^{-h}}
\qquad
\left(
m_0h\ge\log\left(\frac{M}{\tau}\right)+C_0
\right).
\end{align}

Combining \eqref{eq:simp_expansion_variable},
\eqref{eq:tv_approx_variable}, and
\eqref{eq:truncation_variable} gives the final result:

\begin{align*}
\left|
\sqrt{x}\,\frac{F(x)}{G(x)}-1
\right|
&=
\left|
\frac{h}{\pi}
\sum_{k=-m_0}^{m_0}
\operatorname{sech}(kh-y)
-
1
\right|
&&
\text{(by \eqref{eq:simp_expansion_variable})}
\\
&=
\frac{1}{\pi}
\left|
h
\sum_{k=-m_0}^{m_0}
\operatorname{sech}(kh-y)
-
\pi
\right|
\\
&\le
\frac{1}{\pi}
\left|
h
\sum_{k\in\mathbb Z}
\operatorname{sech}(kh-y)
-
\pi
\right|
+
\frac{h}{\pi}
\sum_{|k|>m_0}
\operatorname{sech}(kh-y)
\\
&\le
\frac{2h}{\pi}
+
\frac{4he^{-C_0}}{\pi(1-e^{-h})}
&&
\text{(by \eqref{eq:tv_approx_variable} and
\eqref{eq:truncation_variable})}
\\
&\le
\varepsilon
=
\frac18.
&&
\text{(by \eqref{eq:epsilon_choice_variable})}
\end{align*}
for every \(x\in[\tau^2,M^2]\).

\end{proof}

\section{Improved Multicolor Ramsey Bounds} \label{sec:ramsey}

In this section, we prove the improved geometric lemma (Lemma~\ref{lem:or-composed-geometric-lemma}) and Theorem~\ref{thm:improved-multicolor-ramsey}. We begin by designing the compressor.
The compressor in Theorem~\ref{thm:schrijver-quasi-tensorization} only works on bounded inputs. However, this is not the case for arbitrary mappings and thus does not apply to the proof of the geometric lemma. Moreover, the proof of the geometric lemma requires the growth rate of the compressor scales $\exp(O(\sqrt{x}))$ when $x \geq -1$. We define the following function:
\begin{equation}\label{def:B}
B(x)=\sum_{i=0}^{\infty}\frac{x^i}{(i!)^2}.
\end{equation}
For $x \geq 0$, $B(x) =I_0(2\sqrt x)$,
where $I_0$ is the modified Bessel function of the first kind of order zero.
For \(x<0\), \(B(x)=J_0(2\sqrt {-x})\), where \(J_0\) is the Bessel function of the first kind of order zero. Since $J_0(x)\to 0$ when $x\to \infty$, see, for example~\cite{neuman2004inequalities}, we have the following fact.

\begin{fact}\label{lem:bessel-damping}
There is an absolute constant \(A\ge1\) such that \(|J_0(2\sqrt a)|\le e^{-3}\) for every \(a\ge A\).
\end{fact}
 Let \(A\) be the absolute constant as in Fact~\ref{lem:bessel-damping}.  For \(r\ge2\), choose $m=\left\lceil \frac16\log(8er)\right\rceil$ and set \(\Psi(x)=B(x)^{2m}\). We define the compressor $T(x)$ as
\begin{equation}\label{def:T}
T(x)=\Psi(Ax)-1.
\end{equation}

\begin{lemma} \label{lem:T}
Let $r\geq 2$ be an integer, $A$ be defined as in Fact~\ref{lem:bessel-damping} and the compressor \(T(x)\) defined as in~\eqref{def:T}. $T(x)$ satisfies the following properties:
\begin{enumerate}
    \item \(T(x)\) has nonnegative Taylor coefficients and \(T(x)\geq -1\).
    \item \(T(x)\ge0\) when \(x\ge 0\) and \(T(x)\in[-1, -1+(8er)^{-1}]\) when \(x\le -1\).
    \item For $x\ge -1$, there exists an absolute constant $C_1$ such that
    \[1+T(x)\le \exp\bigl(C_1\log(r)\sqrt{x+1}\bigr).\]
\end{enumerate}
\end{lemma}

\begin{proof}
By definition of $B$ in~\eqref{def:B}, its Taylor coefficients are nonnegative. Therefore, the Taylor coefficients of \(\Psi\) are also nonnegative. Also note that the constant Taylor coefficient of $\Psi(x)$ is $1$ since the constant Taylor coefficient of $B(x)$ is $1$. This gives nonnegativity of all Taylor coefficients of \(T\). And since $\Psi(x) = B(x)^{2m} \geq 0$, we have $T(x) \geq -1$.

If \(x\ge 0\), then since all Taylor coefficients of $B$ are nonnegative and the constant Taylor coefficient of $B$ is $1$, we have \(B(Ax)\ge1\). Thus, we have \(T(x)
\ge0\).  If \(x\le-1\), then \(Ax\le -A\). Thus, by Fact~\ref{lem:bessel-damping}, we have \(\Psi(Ax)\le e^{-6m}\le(8er)^{-1}\), proving the stated bound.

It remains to bound the growth when $x\geq -1$. For \(x\ge0\), by the integral representation of $I_0(x)$ in~\cite{watson1944bessel}, we have $$B(Ax) = I_0(2\sqrt{Ax})=\frac{1}{\pi}\int_0^\pi e^{2\sqrt{Ax} \cos\theta}\,d\theta \le e^{2\sqrt{Ax}}.$$
For \(-1\le x\le0\), one has \(|B(Ax)|=|J_0(2\sqrt{|Ax|})|\le1\).  Therefore, if \(x\ge-1\), we have
\[
1+T(x)=\Psi(Ax) = B(Ax)^{2m}\le \exp(4m\sqrt{A(x+1)})\leq \exp(C_1\log(r) \sqrt{x+1} )
\] for some absolute constant $C_1$.
\end{proof}

We now prove a geometric lemma with better dependence on $r$ that replaces Lemma 3.1 of \cite{balister2024upperboundsmulticolorramsey}.

\begin{proof}[Proof of Lemma~\ref{lem:or-composed-geometric-lemma}]
By Lemma 3.2 of \cite{balister2024upperboundsmulticolorramsey}, for every \((a_1,\ldots,a_r)\in\mathbb Z_{\ge0}^r\), we have
\begin{equation} \label{eq:mm_pos}
\mathbb{E} \prod_{i=1}^r Z_i^{a_i}\ge0.
\end{equation}
The key improvement of our proof lies in our choice of a better special function than the one in~\cite{balister2024upperboundsmulticolorramsey}. In particular, we define
\[
P(x_1,\ldots,x_r)=2r^r Q_r(T(x_1),\ldots,T(x_r)),
\]
where $T$ is defined as in~\eqref{def:T}.
By Lemma \ref{lem:T} and the coefficient positivity of \(Q_r\), the entire function \(P\) has nonnegative Taylor coefficients. Since $X$ is finite, expanding termwise and using equation \eqref{eq:mm_pos} gives that \(\mathbb E P(Z_1,\ldots,Z_r)\ge0\).

Now, define the event \(E=\{Z_i\ge-1\text{ for all }i\}\) and a random variable \(M=\max_i Z_i\).  If \(E\) fails, then some \(Z_j<-1\).  Then, Lemma \ref{lem:T} gives \(T(Z_j)\in[-1,-1+(8er)^{-1}]\) and \(T(Z_i)\ge-1\) for every~\(i\).  Therefore, from Lemma \ref{lem:stren-or}, we know that \(Q_r(T(Z_1),\ldots,T(Z_r))\le -\frac12 r^{-r}\).  Hence \(P(Z)\le -1\) on~\(E^c\).
On the event \(E\), Lemma \ref{lem:T} gives
\[
1+T(Z_i)\le \exp\bigl(C_1\log(r)\sqrt{Z_i+1}\bigr).
\]
Since \(Q_r(t_1,\ldots,t_r) = \prod_{i=1}^r(1+t_i) - \prod_{i=1}^r(1+(1-\frac{1}{r})t_i)\le\prod_{i=1}^r(1+t_i)\) whenever \(t_i\ge-1\), it follows that on \(E\) we have
\[
P(Z)\le 2r^r\exp\bigl(C_1 r\log(r)\sqrt{M+1}\bigr).
\]
Thus, there are constants \(A_r\le\exp(C_3r\log(r))\) and \(C_r\le C_3r\log(r)\) such that \(P(Z)\le A_r e^{C_r\sqrt{M+1}}\) on \(E\), and \(P(Z)\le-1\) on \(E^c\).  Since \(\E P(Z)\ge0\), this gives
\begin{equation}\label{eq:prob-of-Ec}
A_r\E\bigl[e^{C_r\sqrt{M+1}}\1_E\bigr]\ge\Pr(E^c).
\end{equation}
Let \(\beta_r=(4A_r(1+rC_r))^{-1}\).  If \(\mathbb{P}(E)\ge\beta_r\), we can take \(\lambda=-1\) and conclude the proof. So, it suffices to consider the case that \(\mathbb{P}(E)\le\beta_r\). In this case, we have \(\mathbb{P}(E^c)>1/2\). Thus, rearranging \eqref{eq:prob-of-Ec} gives
\begin{equation} \label{eq:rearr}
\E\bigl[e^{C_r\sqrt{M+1}}\1_E\bigr]\ge(2A_r)^{-1}.
\end{equation}

We next prove that the theorem holds for \(C'_r=C_r+1\), then the proof can be concluded by enlarging $C_r$. Suppose for the sake of contradiction that for all $i\in [r]$ and $\lambda \geq -1$, we have
\[
\mathbb{P}\bigl[Z_i\ge\lambda\text{ and } Z_j\ge -1\text{ for every }j\ne i\bigr]
< \beta_r\exp\bigl(-C_r'\sqrt{\lambda+1}\bigr).
\]
Then by union bound, 
\begin{equation}\label{eq:union-bound}
    \mathbb{P}[M\ge\lambda, E]
< r\beta_r e^{-C'_r\sqrt{\lambda+1}}.
\end{equation}
We now compute as follows
\[
\begin{aligned}
\E\!\left[e^{C_r\sqrt{M+1}}\mathbf 1_E\right]
&=
\int_0^\infty
\Pr\!\left[
e^{C_r\sqrt{M+1}}\mathbf 1_E\ge s
\right]\,ds
\\
&=
\int_0^1
\Pr\!\left[
e^{C_r\sqrt{M+1}}\mathbf 1_E\ge s
\right]\,ds
+
\int_1^\infty
\Pr\!\left[
e^{C_r\sqrt{M+1}}\mathbf 1_E\ge s
\right]\,ds.
\end{aligned}\]
For the first integral, the event $\{ e^{C_r\sqrt{M+1}}\mathbf 1_E\ge s\}$ implies the event $E$. Thus,
\[
\int_0^1
\Pr\!\left[
e^{C_r\sqrt{M+1}}\mathbf 1_E\ge s
\right]\,ds\leq 
\Pr(E).\] 
For the second integral, we have
\[
\begin{aligned}
    \int_1^\infty
\Pr\!\left[
e^{C_r\sqrt{M+1}}\mathbf 1_E\ge s
\right]\,ds &= \int_1^\infty
\Pr\!\left[
M\ge \left(\frac{\log s}{C_r}\right)^2-1,\ E
\right]\,ds\\
&< \int_1^\infty
r\beta_r
\exp\!\left(
-C'_r\frac{\log s}{C_r}
\right)\,ds\\
&= r\beta_r C_r,
\end{aligned}
\]
where the inequality follows from~\eqref{eq:union-bound}. Combining the two bounds above, we have
\[
\E[e^{C_r\sqrt{M+1}}\1_E]
< \beta_r+r\beta_r C_r\le (4A_r)^{-1},
\]
which is a contradiction to~\eqref{eq:rearr}, concluding the proof.
\end{proof}

We now plug in our improved geometric lemma into the algorithm of \cite{balister2024upperboundsmulticolorramsey}. In particular, tracing their analysis with
the dependence on the geometric parameters tracked gives the following lemma.

\begin{lemma}
\label{lem:ramsey-bookkeeping}
Fix \(r\ge 2\). Suppose that~\eqref{eq:ramsey_lemma} holds with
parameters \(C_r\ge 1\) and \(0<\beta_r\le 1\) and there
exists a parameter \(\mathcal M_r\ge 2\) such that
\[
    C_r^2\le \mathcal M_r,
    \qquad
    \log\frac r{\beta_r}\le \mathcal M_r\log(2r),
    \qquad
    \log(4r\mathcal M_r^2)\le K_0\log(2r)
\]
for some absolute constant $K_0 > 0$.
Then, there are absolute constants $c,K > 0$ such that
\[
    R_r(k)
    \le
    \exp\!\left(-\frac{ck}{r^3\mathcal M_r^3}\right)r^{rk}
\]
for every $k\ge Kr^2\mathcal M_r^6.$
\end{lemma}

We note that the proof of Lemma~\ref{lem:ramsey-bookkeeping} follows directly from \cite{balister2024upperboundsmulticolorramsey} with no new ingredients. As such, we defer the proof to the appendix. Now, the proof of Theorem \ref{thm:improved-multicolor-ramsey} follows immediately. 

\begin{proof}[Proof of Theorem~\ref{thm:improved-multicolor-ramsey}]
By Lemma~\ref{lem:or-composed-geometric-lemma}, equation~\eqref{eq:ramsey_lemma}
holds with \(C_r\le C r\log(r)\) and
\(\beta_r\ge \exp\{-C r\log(r)\}\), where \(C>0\) is an absolute constant.
Set $\mathcal M_r=K_1 r^2(\log(2r))^2,$ where \(K_1\) is a fixed sufficiently large absolute constant. Then
\[
    C_r^2\le C^2r^2(\log(r))^2\le \mathcal M_r,
\]
\[
    \log\frac r{\beta_r}
    \le \log r+Cr\log(r)
    \le K_1r^2(\log(2r))^3
    =\mathcal M_r\log(2r),
\]
and
\[
    \log(4r\mathcal M_r^2)
    =
    \log\!\left(4K_1^2r^5(\log(2r))^4\right)
    \le K_0\log(2r)
\]
for an absolute constant \(K_0\). Thus, Lemma~\ref{lem:ramsey-bookkeeping}
gives
\[
    R_r(k)
    \le
    \exp\!\left(-\frac{c k}{r^3\mathcal M_r^3}\right)r^{rk} \le 
    \exp\!\left(
        -\frac{c k}{r^9(\log(2r))^6}
    \right)r^{rk}
\]
for every \(k\ge K r^2\mathcal M_r^6 \ge C_0r^{14}(\log(2r))^{12}.\)
\end{proof}

\section{Conclusion} \label{sec:conclusion}

In this paper, we introduced a robust polynomial framework for OR-type combinatorial problems. The general workflow is summarized in Figure~\ref{fig:or-polynomial-proof-strategy}. Through a carefully designed compressor, the framework takes in one-coordinate positive semidefinite certificates, compresses the input, and composes with the robust OR polynomial~\eqref{def:robust-or} 
to produce PSD certificates for OR-type problems. 

We then demonstrated the power of our framework by two concrete applications. The first one is a new quasi-tensorization theorem for the Schrijver number (Theorem~\ref{thm:schrijver-quasi-tensorization}). This gives a significantly tighter bound on $M_r(n)$ and $M^{\pm}_r(n)$ compared to the multiplicative \lovasz theta number (Corollary~\ref{cor:Mr(n)}).
The second application was an improvement on the geometric lemma in~\cite{balister2024upperboundsmulticolorramsey},
giving an improvement to known upper bounds on the multicolor Ramsey number $R_r(k)$ (Lemma~\ref{lem:or-composed-geometric-lemma}, Theorem~\ref{thm:improved-multicolor-ramsey}). 

Some interesting future directions and problems are as follows. The first conjecture concerns a tighter bound on the Schrijver number of $G_n^{\boxtimes 2}$.

\begin{conjecture}
    There exists an absolute constant $C > 0$ such that $\vartheta'(G_n^{\boxtimes 2}) \le n^C$.
\end{conjecture}
The above conjecture is supported by numerical evidence. We observed that the correct order of $M_2(n)$ seems to be $n^{2+o(1)}$. This may potentially be attacked via a tighter analysis of the Delsarte LP. Another related future direction is on a stronger quasi-tensorization theorem for the Schrijver number.

\begin{problem}
    Let $G_1 = (V_1, E_1), \ldots, G_r = (V_r, E_r)$ be graphs and suppose $\vartheta'(G_i) \geq 2$ for all $i\in [r]$. 
    Is there a constant $C_r$ such that
\[
    \vartheta'(G_1\boxtimes \ldots \boxtimes G_r) \leq \prod_{i=1}^r \vartheta'(G_i)^{C_r}.
\] 
\end{problem}

\section{Acknowledgments}

We would like to thank Noga Alon for introducing us to this problem, and Pedro Paredes and Yuval Wigderson for feedback on initial stages of this work.

\bibliographystyle{alpha}
\bibliography{ref}

\section*{AI Use Statement}

Artificial intelligence assistance (ChatGPT 5.5 Pro) was used to verify the correctness of mathematical proofs, especially the parameter bookkeeping in the proof of Lemma \ref{lem:ramsey-bookkeeping}.

\appendix
\section{Appendix}

\subsection{Proof of Lemma \ref{lem:stren-or}}

\begin{proof}[Proof of Lemma~\ref{lem:stren-or}]
Set $\lambda =  1 - \frac{1}{r},$ \(u_i=1+t_i\) and \(v_i=1+\lambda t_i\).  Since \(t_i\ge-1\), we have \(u_i\ge0\) and \(v_i\ge1-\lambda=1/r\).  Moreover, \((1+t)/(1+\lambda t)\le1/\lambda\) for all \(t\ge-1\), because this ratio is increasing on \([-1,\infty)\) and $\lim_{t \rightarrow \infty} \frac{1+t}{1 + \lambda t} = \frac{1}{\lambda}$.  Denote $j$ to be one of the coordinates satisfying  \(-1\le t_j\le -1+\eta\). Note that we must have that \(u_j\le\eta\) and \(v_j\ge1/r\).  Therefore, we have that
\begin{align*}
Q_r(t)
&= \prod_{i=1}^r u_i-\prod_{i=1}^r v_i \\
&=
\left(
    \frac{\prod_{i=1}^r u_i}{\prod_{i=1}^r v_i}
    -1
\right)
\prod_{i=1}^r v_i \\
&\le
\left(\eta r\lambda^{-(r-1)}-1\right)
\prod_{i=1}^r v_i
&& \text{(by the ratio bound)} \\
&\le
\left(\eta er-1\right)
\prod_{i=1}^r v_i
&& \text{(since } \lambda^{-(r-1)}\le e) \\
&\le
-\frac78 \prod_{i=1}^r v_i
&& \text{(since } \eta er\le \frac18) \\
&\le
-\frac78 r^{-r},
&& \text{(since } v_i\ge \frac1r \text{ for all } i)
\end{align*}
giving the result.
\end{proof}
\subsection{Proof of Lemma \ref{lem:ramsey-bookkeeping}}

We first introduce some notation. An \(r\)-coloring of \(E(K_n)\) is a map
\(\chi:E(K_n)\to[r]\). Given a vertex \(v\), its color-\(i\) neighborhood is
\[
    N_i(v):=\{u:\chi(\{u,v\})=i\}.
\]
A book \((A,B)\) is the graph on vertex set \(A\cup B\) obtained from the
clique on \(A\cup B\) by deleting all edges internal to \(B\). We call \(A\)
the spine and \(B\) the page set. A \((t,m)\)-book in color \(i\) is a book
\((A,B)\) with \(|A|=t\), \(|B|=m\), and all edges of the book having color
\(i\).

Given two vertex sets \(X,Y\), define the minimum color-\(i\) density from
\(X\) to \(Y\) by
\[
    p_i(X,Y):=\min_{x\in X}\frac{|N_i(x)\cap Y|}{|Y|}.
\]
The key application of the geometric lemma is towards Lemma 2.2 in~\cite{balister2024upperboundsmulticolorramsey}. Parametrized, it is as follows:

\begin{lemma}[Parametrized Lemma 2.2 of~\cite{balister2024upperboundsmulticolorramsey}]
\label{lem:param-key}
Suppose that equation \eqref{eq:ramsey_lemma} (the geometric lemma) holds with constants $C_r$ and $\beta_r$. Let \(\chi\) be an \(r\)-coloring of \(E(K_n)\), let
\(X,Y_1,\ldots,Y_r \subset V(K_n)\) be non-empty sets of vertices,
and let \(\alpha_1,\ldots,\alpha_r > 0\). There exists a vertex
\(x \in X\), a color \(\ell \in [r]\), sets \(X' \subset X\) and
\(Y_1',\ldots,Y_r'\) with \(Y_i' \subset N_i(x) \cap Y_i\) for each
\(i \in [r]\), and \(\lambda \geq -1\), such that
\[
|X'| \geq \beta_r e^{-C_r\sqrt{\lambda+1}} |X|
\qquad \text{and} \qquad
p_\ell(X',Y_\ell') \geq p_\ell(X,Y_\ell) + \lambda \alpha_\ell,
\]
and moreover
\[
|Y_i'| = p_i(X,Y_i)|Y_i|
\qquad \text{and} \qquad
p_i(X',Y_i') \geq p_i(X,Y_i) - \alpha_i
\]
for every \(i \in [r]\).
\end{lemma}

Having established the above lemma,
\cite{balister2024upperboundsmulticolorramsey} proceeds by running a book
algorithm. We will not describe the algorithm in full, but only recall the
notation and the properties that we need. The book algorithm takes as input an \(r\)-coloring \(\chi\), sets
\(X,Y_1,\ldots,Y_r\), and parameters \(t\in\mathbb N\), \(\delta>0\), and
\(\lambda_0\ge 0\). Set
\[
    p_0=\min_{i\in[r]}p_i(X,Y_i).
\]
After \(s\) steps of the algorithm, we denote the reservoir set by \(X(s)\),
the page sets by \(Y_1(s),\ldots,Y_r(s)\), and the color-\(i\) spine by
\(T_i(s)\). The algorithm initializes $T_1(0)=\cdots=T_r(0)=\emptyset,$
    $X(0)=X,$ and  $Y_i(0)=Y_i.$
We also write
\[
    p_i(s)=p_i(X(s),Y_i(s)),
    \qquad
    \alpha_i(s)=\frac{p_i(s)-p_0+\delta}{t}.
\]

At each non-terminal step, the algorithm applies Lemma~\ref{lem:param-key} to
the current sets \(X(s),Y_1(s),\ldots,Y_r(s)\) with parameters
\(\alpha_1(s),\ldots,\alpha_r(s)\). This produces a vertex \(x(s)\), a color
\(\ell(s)\), updated sets, and a parameter \(\lambda(s)\ge -1\). If
\(\lambda(s)>\lambda_0\), the algorithm performs a density-boost step in color
\(\ell(s)\). Otherwise, it performs a color step, adding \(x(s)\) to one of the
spines. The algorithm stops only when either \(X(s)=\emptyset\), or $\max_{i\in[r]} |T_i(s)|=t.$

For each \(i\in[r]\), let $B_i(s)
    :=
    \{0\le j<s:\ell(j)=i\text{ and }\lambda(j)>\lambda_0\},$
and set $B(s)=B_1(s)\cup\cdots\cup B_r(s).$ The following lemma establishes the properties of the algorithm.
\begin{lemma}
\label{lem:book-algorithm-package}
Run the multicolor book algorithm of \cite{balister2024upperboundsmulticolorramsey}. Then the following statements hold.

\begin{enumerate}
    \item For every \(i\in[r]\) and \(s\in\mathbb N\), we have that $
        p_i(s)-p_0+\delta
        \ge
        \delta\left(1-\frac1t\right)^t
        \prod_{j\in B_i(s)}
        \left(1+\frac{\lambda(j)}{t}\right).
    $

    \item If \(t\ge 2\), then
    $
        p_i(s)\ge p_0-\frac{3\delta}{4}
        \ \text{and}\
        \alpha_i(s)\ge \frac{\delta}{4t}
    $
    for every \(i\in[r]\) and \(s\in\mathbb N\).

    \item If \(t\ge \lambda_0\ge 2\) and \(\delta\le 1/4\), then
    $
        |B_i(s)|
        \le
        \frac{4\log(1/\delta)}{\lambda_0}\,t
    $
    for every \(i\in[r]\) and \(s\in\mathbb N\).

    \item If \(t\ge 2\), then
    $
        |Y_i(s)|
        \ge
        \left(p_0-\frac{3\delta}{4}\right)^{t+|B_i(s)|}
        |Y_i(0)|
    $
    for every \(i\in[r]\) and \(s\in\mathbb N\).

    \item Let $
        \eta=\frac{\beta_r}{r}e^{-C_r\sqrt{\lambda_0+1}}.$
    Then
    $
        |X(s)|
        \ge
        \eta^{rt+|B(s)|}
        \exp\left(
            -C_r\sum_{j\in B(s)}\sqrt{\lambda(j)+1}
        \right)|X(0)|-rt
    $
    for every \(s\in\mathbb N\).

    \item If \(t\ge \lambda_0/\delta>0\) and \(\delta\le 1/4\), then $
        \sum_{j\in B(s)}\sqrt{\lambda(j)}
        \le
        \frac{7r\log(1/\delta)}{\sqrt{\lambda_0}}\,t
    $
    for every \(s\in\mathbb N\).
\end{enumerate}
\end{lemma}

\begin{proof}
These are Lemmas 4.1--4.6 of
\cite{balister2024upperboundsmulticolorramsey}. The proofs are unchanged after replacing their
constants \(C,\beta\) by \(C_r,\beta_r\).
\end{proof}

\begin{lemma}
\label{lem:param-book}
Assume that equation~\eqref{eq:ramsey_lemma} holds with parameters
\(C_r\ge 1\) and \(0<\beta_r\le 1\). Let \(0<p_r\le 1\), let
\(\mu_r\ge 2^{12}C_r^2\), and suppose that
\[
    \log\frac r{\beta_r}
    \le
    \frac{\mu_r}{16}\log\frac{\mu_r^2}{p_r}.
\]
Let \(t_r,m_r\in\mathbb N\) satisfy $t_r\ge \frac{\mu_r^5}{p_r^2}.$
Furthermore, let \(\chi\) be an \(r\)-coloring, and
\(X,Y_1,\ldots,Y_r\) be non-empty vertex sets such that $p_i(X,Y_i)\ge p_r$
for every $i\in[r].$

If $|X|\ge \left(\frac{\mu_r^2}{p_r}\right)^{\mu_r rt_r}$
and $|Y_i|
    \ge
    \left(
        \frac{\exp(2^9C_r^2/\mu_r^2)}{p_r}
    \right)^{t_r} m_r$ for every $i\in[r],$ then \(\chi\) contains a monochromatic \((t_r,m_r)\)-book.
\end{lemma}

\begin{proof}
Write $p_{0,r}=\min_{i\in[r]}p_i(X,Y_i).$ By the density hypothesis, \(p_{0,r}\ge p_r\). Run the multicolor book
algorithm with
\[
    \delta_r=\frac{p_r}{\mu_r^2},
    \qquad
    \lambda_{0,r}
    =
    \left(\frac{\mu_r\log(1/\delta_r)}{8C_r}\right)^2.
\]
Since \(p_r\le 1\), \(C_r\ge 1\), and \(\mu_r\ge 2^{12}C_r^2\), we have
$\delta_r\le \frac14,$ and $\lambda_{0,r}\ge 2.$ Moreover, using \(\log x\le 2x^{1/4}\) for \(x\ge 1\),
\[
    \frac{\lambda_{0,r}}{\delta_r}
    =
    \frac{\mu_r^4\log^2(1/\delta_r)}{64C_r^2p_r}
    \le
    \frac{\mu_r^5}{p_r^2}
    \le t_r.
\]
In particular, \(t_r\ge \lambda_{0,r}\). By Lemma~\ref{lem:book-algorithm-package}(3), for every \(i\in[r]\) and
every \(s\),
\[
    |B_i(s)|
    \le
    \frac{4\log(1/\delta_r)}{\lambda_{0,r}}t_r
    =
    \frac{256C_r^2}{\mu_r^2\log(1/\delta_r)}t_r
    \le t_r.
\]
Thus $|B(s)|\le rt_r.$ We now show that the page sets never become too small. By
Lemma~\ref{lem:book-algorithm-package}(4), \(p_{0,r}\ge p_r\), and the lower
bound on \(|Y_i|\),
\[
\begin{aligned}
    |Y_i(s)|
    &\ge
    \left(p_r-\frac{3\delta_r}{4}\right)^{t_r+|B_i(s)|}
    \left(
        \frac{\exp(2^9C_r^2/\mu_r^2)}{p_r}
    \right)^{t_r}m_r  \\
    &=
    p_r^{|B_i(s)|}
    \left(1-\frac{3}{4\mu_r^2}\right)^{t_r+|B_i(s)|}
    \exp(2^9C_r^2t_r/\mu_r^2)m_r  \\
    &\ge
    e^{-2t_r/\mu_r^2}
    p_r^{|B_i(s)|}
    \exp(2^9C_r^2t_r/\mu_r^2)m_r.
\end{aligned}
\]
Since \(\log(1/p_r)\le \log(1/\delta_r)\), the bound on \(|B_i(s)|\) gives
\[
    p_r^{|B_i(s)|}
    \ge
    \exp(-256C_r^2t_r/\mu_r^2).
\]
Therefore
\[
    |Y_i(s)|
    \ge
    \exp\left(\frac{(2^9-256)C_r^2-2}{\mu_r^2}t_r\right)m_r
    \ge m_r
\]
for every \(i\in[r]\) and every \(s\). It remains to show that \(X(s)\) never becomes empty before the algorithm
reaches a spine of size \(t_r\). Set $\eta_r=\frac{\beta_r}{r}e^{-C_r\sqrt{\lambda_{0,r}+1}}.$ By Lemma~\ref{lem:book-algorithm-package}(5) and \(|B(s)|\le rt_r\),
\[
    |X(s)|
    \ge
    \eta_r^{2rt_r}
    \exp\left(
        -C_r\sum_{j\in B(s)}\sqrt{\lambda(j)+1}
    \right)|X|-rt_r.
\]
Now
\[
    \log\frac r{\beta_r}
    \le
    \frac{\mu_r}{16}\log(1/\delta_r),
    \qquad
    C_r\sqrt{\lambda_{0,r}+1}
    \le
    \frac{\sqrt2\,\mu_r}{8}\log(1/\delta_r),
\]
so
\[
    \eta_r^{2rt_r}
    \ge
    \exp\left(-\frac{\mu_r rt_r}{2}\log(1/\delta_r)\right)
    =
    \left(\frac{\mu_r^2}{p_r}\right)^{-\mu_r rt_r/2}.
\]
On the other hand, Lemma~\ref{lem:book-algorithm-package}(6) gives
\[
    \sum_{j\in B(s)}\sqrt{\lambda(j)}
    \le
    \frac{7r\log(1/\delta_r)}{\sqrt{\lambda_{0,r}}}t_r
    =
    \frac{56C_rrt_r}{\mu_r}.
\]
Since every \(j\in B(s)\) satisfies \(\lambda(j)>\lambda_{0,r}\ge 1\), it follows that
\[
    \sum_{j\in B(s)}\sqrt{\lambda(j)+1}
    \le
    2\sum_{j\in B(s)}\sqrt{\lambda(j)}
    \le
    \frac{2^7C_rrt_r}{\mu_r}.
\]
Combining the last two estimates with the lower bound on \(|X|\), we obtain
\[
    |X(s)|
    \ge
    \left(\frac{\mu_r^2}{p_r}\right)^{\mu_r rt_r/2}
    \exp\left(-\frac{2^7C_r^2rt_r}{\mu_r}\right)
    -rt_r.
\]
Since \(\mu_r\ge 2^{12}C_r^2\) and \(\log(\mu_r^2/p_r)\ge 1\), the first term is
at least \(e^{rt_r}\). Hence the right-hand side is positive, and therefore
$|X(s)|>0$ for every \(s\).

Thus the algorithm cannot terminate because \(X(s)\) becomes empty. Before
termination there are at most \(rt_r\) color steps, and by the bound above
there are at most \(rt_r\) density-boost steps. Hence the algorithm must
terminate with $\max_{i\in[r]}|T_i|=t_r.$ 

Now, choose \(i\in[r]\) with \(|T_i|=t_r\). By construction, \(T_i\) is a
monochromatic clique in color \(i\), and $Y_i(s)\subseteq \bigcap_{x\in T_i}N_i(x).
$
Since \(|Y_i(s)|\ge m_r\), choosing any subset \(M\subseteq Y_i(s)\) of size
\(m_r\), the pair \((T_i,M)\) is a monochromatic \((t_r,m_r)\)-book in color
\(i\).
\end{proof}

We now deduce a Ramsey bound from the books. The following lemma collects the
Ramsey-side bookkeeping estimates from
\cite{balister2024upperboundsmulticolorramsey}. Since these estimates do not
depend on \(C_r\) or \(\beta_r\), they require no modification here. For positive integers \(k_1,\ldots,k_r\), let \(R_r(k_1,\ldots,k_r)\) denote
the multicolor Ramsey number, that is, the least \(N\) such that every
\(r\)-coloring of \(E(K_N)\) contains a monochromatic \(K_{k_i}\) in color
\(i\), for some \(i\in[r]\).

\begin{lemma}
\label{lem:ramsey-final-package}
The following Lemmas from \cite{balister2024upperboundsmulticolorramsey} hold.

\begin{enumerate}
    \item (Lemma 5.2 of
    \cite{balister2024upperboundsmulticolorramsey}). Let
    \(n,r\in\mathbb N\), \(0<\varepsilon<1\), and let \(\chi\) be an
    \(r\)-coloring of \(E(K_n)\). There exist disjoint sets of vertices
    \(S_1,\ldots,S_r,W\) such that, writing
    $
        s=|S_1|+\cdots+|S_r|,
    $
    we have
    $|W|\ge \left(\frac{1+\varepsilon}{r}\right)^s n,$
    and
    $|N_i(w)\cap W|
        \ge
        \left(\frac1r-\varepsilon\right)|W|-1
    $
    for every \(w\in W\) and every \(i\in[r]\). Moreover, \((S_i,W)\) is a
    monochromatic book in color \(i\) for every \(i\in[r]\).

    \item (Lemma 5.3 of
    \cite{balister2024upperboundsmulticolorramsey}). If \(k,r\ge 2\),
    \(0<\varepsilon<1\), and \(s_1,\ldots,s_r\in\{0,\ldots,k-1\}\) satisfy
    $
        s=s_1+\cdots+s_r\ge \varepsilon^2 k,
    $
    then
    \[
        R_r(k-s_1,\ldots,k-s_r)
        \le
        e^{-\varepsilon^3k/2}
        \left(\frac{1+\varepsilon}{r}\right)^s
        r^{rk}.
    \]

    \item (Lemma A.1 of
    \cite{balister2024upperboundsmulticolorramsey}, together with the
    Erd\H{o}s--Szekeres multinomial bound). For every \(r\ge 2\) and
    \(3\le t\le k\),
    \[
        R_r(k,\ldots,k,k-t)
        \le
        e^{-t^2/6k}r^{rk-t}.
    \]
\end{enumerate}
\end{lemma}

With the above lemma, we are now able to complete the proof of Lemma \ref{lem:ramsey-bookkeeping}.

\begin{proof}[Proof of Lemma~\ref{lem:ramsey-bookkeeping}]
Choose absolute constants as follows. Take \(K_0\) sufficiently large, then take
\(A\) sufficiently large so that the interval
\[
    \left[
        \frac{2^{15}}{A^2},
        \frac{1}{100A\left(K_0+\frac{\log(2A^2)}{\log 4}+2\right)}
    \right]
\]
is non-empty. Choose \(b\) in this interval, choose \(a>0\) with
\(a\le b/1000\), choose \(c>0\) with
\(c\le \min\{a^3/4,b^2/100\}\), and finally choose \(K\) sufficiently large in
terms of \(A\) and \(b\).

It is enough, after harmless changes to \(c\) and \(K\) to absorb integer
rounding, to show that every \(r\)-coloring of \(K_n\) with
\[
    n\ge
    \exp\left(-\frac{ck}{r^3\mathcal M_r^3}\right)r^{rk}
\]
contains a monochromatic \(K_k\).

Let \(\chi\) be such a coloring. Set
\(\gamma_r=c/(r^3\mathcal M_r^3)\) and
\(\varepsilon_r=a/(r\mathcal M_r)\). Since \(c\le a^3/4\), we have
\(\gamma_r\le \varepsilon_r^3/2\). Apply
Lemma~\ref{lem:ramsey-final-package}(1) to obtain disjoint sets
\(S_1,\ldots,S_r,W\). Write \(s_i=|S_i|\) and
\(s=s_1+\cdots+s_r\). If some \(s_i\ge k\), then \(S_i\) already contains a
monochromatic \(K_k\), so we may assume \(s_i<k\) for every \(i\in[r]\).

If \(s\ge \varepsilon_r^2k\), then Lemma~\ref{lem:ramsey-final-package}(2)
gives
\[
    R_r(k-s_1,\ldots,k-s_r)
    \le
    e^{-\varepsilon_r^3k/2}
    \left(\frac{1+\varepsilon_r}{r}\right)^s
    r^{rk}.
\]
Since \(\gamma_r\le\varepsilon_r^3/2\), the lower bound on \(W\) gives
\[
    |W|
    \ge
    \left(\frac{1+\varepsilon_r}{r}\right)^s
    e^{-\gamma_r k}r^{rk}
    \ge
    R_r(k-s_1,\ldots,k-s_r).
\]
Thus \(W\) contains a monochromatic \(K_{k-s_i}\) in some color \(i\), and
because \((S_i,W)\) is a monochromatic book in color \(i\), this extends to a
monochromatic \(K_k\). Hence we may assume from now on that
\(s<\varepsilon_r^2k\).

Set \(X=Y_1=\cdots=Y_r=W\), and define
\[
    p_r=\frac1r-2\varepsilon_r,\qquad
    \mu_r=A\mathcal M_r,\qquad
    t_r=\left\lfloor\frac{bk}{\mathcal M_r}\right\rfloor,\qquad
    m_r=R_r(k,\ldots,k,k-t_r).
\]
By the choice of \(K\),
\[
    \frac{bk}{2\mathcal M_r}\le t_r\le \frac{bk}{\mathcal M_r},
    \qquad
    3\le t_r\le k,
    \qquad
    \varepsilon_r^2k\le \varepsilon_r t_r.
\]
Since \(a\) is small and \(\mathcal M_r\ge2\), we have \(p_r\ge 1/(2r)\).
Moreover,
\[
    |W|
    \ge
    \left(\frac{1+\varepsilon_r}{r}\right)^s n
    \ge
    r^{-\varepsilon_r^2k}
    \exp\left(-\frac{ck}{r^3\mathcal M_r^3}\right)r^{rk}
    \ge r^{rk/2},
\]
and hence, by increasing \(K\) if necessary, \(|W|\ge 1/\varepsilon_r\).
Therefore Lemma~\ref{lem:ramsey-final-package}(1) gives
\[
    |N_i(w)\cap W|
    \ge
    \left(\frac1r-\varepsilon_r\right)|W|-1
    \ge
    \left(\frac1r-2\varepsilon_r\right)|W|
    =
    p_r|W|
\]
for every \(w\in W\) and every \(i\in[r]\). Thus \(p_i(X,Y_i)\ge p_r\) for
every \(i\in[r]\).

We check the hypotheses of Lemma~\ref{lem:param-book}. First,
\(\mu_r=A\mathcal M_r\ge 2^{12}C_r^2\), since \(C_r^2\le\mathcal M_r\) and
\(A\) is large. Also,
\[
    \log\frac r{\beta_r}
    \le
    \mathcal M_r\log(2r)
    \le
    \frac{\mu_r}{16}\log\frac{\mu_r^2}{p_r},
\]
because \(\mu_r^2/p_r\ge A^2\mathcal M_r^2r\) and \(A\) is large. Finally, we have that $\frac{\mu_r^5}{p_r^2}
    \le
    4A^5r^2\mathcal M_r^5
    \le t_r$ because of the fact that \(k\ge Kr^2\mathcal M_r^6\) and the choice of \(K\).

For the reservoir condition, using \(p_r\ge 1/(2r)\), \(t_r\le
bk/\mathcal M_r\), and \(\mu_r=A\mathcal M_r\), we have
\[
    \left(\frac{\mu_r^2}{p_r}\right)^{\mu_r rt_r}
    \le
    \exp\left(Ab\,rk\log(2A^2r\mathcal M_r^2)\right).
\]
The hypothesis \(\log(4r\mathcal M_r^2)\le K_0\log(2r)\), together with the
choice of \(b\), implies that this is at most \(r^{rk/4}\). Since
\(|W|\ge r^{rk/2}\), the required lower bound on \(|X|\) follows.

It remains to check the page-set condition. By
Lemma~\ref{lem:ramsey-final-package}(3),
\[
    m_r
    =
    R_r(k,\ldots,k,k-t_r)
    \le
    e^{-t_r^2/6k}r^{rk-t_r}.
\]
Since \(c\le b^2/100\) and \(t_r\ge bk/(2\mathcal M_r)\), we have $
    \frac{ck}{r^3\mathcal M_r^3}\le \frac{t_r^2}{24k},$
and hence
\[
    n
    \ge
    \exp\left(-\frac{ck}{r^3\mathcal M_r^3}\right)r^{rk}
    \ge
    r^{t_r}e^{t_r^2/8k}m_r.
\]
Also \(p_r=(1-2\varepsilon_r r)/r\ge e^{-4\varepsilon_r r}/r\), so
\(r^{t_r}\ge p_r^{-t_r}e^{-4\varepsilon_r rt_r}\). Therefore
\[
    n
    \ge
    m_rp_r^{-t_r}
    \exp\left(\frac{t_r^2}{8k}-4\varepsilon_r rt_r\right).
\]
Since \(s<\varepsilon_r^2k\le\varepsilon_r t_r\), Lemma~\ref{lem:ramsey-final-package}(1)
gives
\[
    |W|
    \ge
    \left(\frac{1+\varepsilon_r}{r}\right)^{\varepsilon_r t_r}n
    \ge
    e^{-\varepsilon_r rt_r}n 
    \ge
    m_rp_r^{-t_r}
    \exp\left(\frac{t_r^2}{8k}-5\varepsilon_r rt_r\right).
\]
Finally, using \(C_r^2\le\mathcal M_r\), \(\mu_r=A\mathcal M_r\), and
\(t_r\ge bk/(2\mathcal M_r)\),
\[
    \frac{t_r}{8k}
    \ge
    \frac{b}{16\mathcal M_r}
    \ge
    \frac{5a}{\mathcal M_r}
    +
    \frac{2^9}{A^2\mathcal M_r}
    \ge
    5\varepsilon_r r+\frac{2^9C_r^2}{\mu_r^2}.
\]
Thus
\[
    |W|
    \ge
    \left(
        \frac{\exp(2^9C_r^2/\mu_r^2)}{p_r}
    \right)^{t_r}m_r.
\]
All hypotheses of Lemma~\ref{lem:param-book} are now satisfied with
\(p=p_r\), \(\mu=\mu_r\), \(t=t_r\), and \(m=m_r\).

Therefore \(\chi[W]\) contains a monochromatic \((t_r,m_r)\)-book. Suppose it
has color \(i\), spine \(T\), and page set \(P\). Since
\(|P|=m_r=R_r(k,\ldots,k,k-t_r)\), the coloring induced on \(P\) contains
either a \(K_{k-t_r}\) in color \(i\), or a \(K_k\) in some color \(j\ne i\).
In the latter case we are done. In the former case, the \(K_{k-t_r}\) together
with the spine \(T\) forms a monochromatic \(K_k\) in color \(i\) as desired.
\end{proof}

\end{document}

%% file: shortcuts.tex
\usepackage{xcolor}

\newcommand{\lovasz}{Lov\a' asz\space}

\newcommand{\Hyp}{\mathcal{H}_n}

\newcommand{\ip}[2]{\left\langle #1,#2\right\rangle}

\newcommand{\x}[1]{x^{(#1)}}

\newcommand{\sphere}{S^{n-1}}
\newcommand{\symm}{\mathbb{S}^n}

\newcommand{\E}{\mathbb{E}}
\renewcommand{\Pr}{\mathbb{P}}

\newcommand{\productkernel}[1]{\tilde{K}_{#1}}

\newcommand{\orhypercube}{f}

%% file: ref.bib
@article {balister2024upperboundsmulticolorramsey,
    AUTHOR = {Balister, Paul and Bollob\'as, B\'ela and Campos, Marcelo and
              Griffiths, Simon and Hurley, Eoin and Morris, Robert and
              Sahasrabudhe, Julian and Tiba, Marius},
     TITLE = {Upper bounds for multicolour {R}amsey numbers},
   JOURNAL = {Journal of the American Mathematical Society},
  FJOURNAL = {Journal of the American Mathematical Society},
    VOLUME = {39},
      YEAR = {2026},
    NUMBER = {3},
     PAGES = {765--780},
      ISSN = {0894-0347,1088-6834},
   MRCLASS = {05D10 (05C35 05D40)},
  MRNUMBER = {5057671},
       DOI = {10.1090/jams/1069},
       URL = {https://doi.org/10.1090/jams/1069},
}

@book{schrijver1986theory,
  author    = {Alexander Schrijver},
  title     = {Theory of Linear and Integer Programming},
  publisher = {Wiley},
  year      = {1986}
}

@article{delsarte1973algebraic,
  title={An algebraic approach to the association schemes of coding theory},
  author={Delsarte, Philippe},
  journal={Philips Res. Rep. Suppl.},
  volume={10},
  pages={vi+--97},
  year={1973}
}

@article{DelsarteGoethalsSeidel1977,
  author  = {Delsarte, Philippe and Goethals, Jean-Marie and Seidel, Johan Jacob},
  title   = {Spherical codes and designs},
  journal = {Geometriae Dedicata},
  volume  = {6},
  number  = {3},
  pages   = {363--388},
  year    = {1977},
  doi     = {10.1007/BF03187604}
}

@book{MacWilliamsSloane1977,
  author    = {MacWilliams, Florence Jessie and Sloane, Neil James Alexander},
  title     = {The Theory of Error-Correcting Codes},
  series    = {North-Holland Mathematical Library},
  volume    = {16},
  publisher = {North-Holland},
  address   = {Amsterdam},
  year      = {1977},
  isbn      = {0444850090}
}

@book{vanLint1999,
  author    = {van Lint, Jacobus Hendricus},
  title     = {Introduction to Coding Theory},
  edition   = {3},
  series    = {Graduate Texts in Mathematics},
  volume    = {86},
  publisher = {Springer},
  address   = {Berlin},
  year      = {1999},
  doi       = {10.1007/978-3-642-58575-3}
}

@article{Lovasz1979,
  author  = {Lov{\'a}sz, L{\'a}szl{\'o}},
  title   = {On the {S}hannon capacity of a graph},
  journal = {IEEE Transactions on Information Theory},
  volume  = {25},
  number  = {1},
  pages   = {1--7},
  year    = {1979},
  doi     = {10.1109/TIT.1979.1055985}
}

@article{Schrijver1979,
  author  = {Schrijver, Alexander},
  title   = {A comparison of the {D}elsarte and {L}ov{\'a}sz bounds},
  journal = {IEEE Transactions on Information Theory},
  volume  = {25},
  number  = {4},
  pages   = {425--429},
  year    = {1979},
  doi     = {10.1109/TIT.1979.1056072}
}

@article{DeKlerkPasechnik2007,
  author  = {de Klerk, Etienne and Pasechnik, Dmitrii V.},
  title   = {A note on the stability number of an orthogonality graph},
  journal = {European Journal of Combinatorics},
  volume  = {28},
  number  = {7},
  pages   = {1971--1979},
  year    = {2007},
  doi     = {10.1016/j.ejc.2006.08.011}
}

@article{Rankin1955,
  author  = {Rankin, Robert Alexander},
  title   = {The Closest Packing of Spherical Caps in \(n\) Dimensions},
  journal = {Proceedings of the Glasgow Mathematical Association},
  volume  = {2},
  pages   = {139--144},
  year    = {1955},
  doi     = {10.1017/S2040618500033219}
}

@incollection{godsil2016,
  author    = {Chris Godsil and Karen Meagher},
  title = {Erd{\H{o}}s--Ko--Rado Theorems: Algebraic Approaches},
  publisher = {Cambridge University Press},
  year      = {2016},
  address   = {Cambridge},
}

@misc{NarangJu2025,
  author        = {Narang, Ijay and Ju, Muchen},
  title         = {Sharp Inner Product Correlations for Hypercube Bijections},
  year          = {2025},
  eprint        = {2509.00716},
  archivePrefix = {arXiv},
  primaryClass  = {math.CO},
  url           = {https://arxiv.org/abs/2509.00716}
}

@unpublished{samorodnitsky1998extremal,
  author = {Samorodnitsky, Alex},
  title  = {Extremal properties of solutions for {Delsarte}'s linear program},
  note   = {Unpublished manuscript, available at \url{http://www.cs.huji.ac.il/~salex/papers/old_sq_measure.ps}},
  year   = {1998},
  url    = {http://www.cs.huji.ac.il/~salex/papers/old_sq_measure.ps}
}

@article{liu2026smoothed,
  title={Smoothed Score Queries and the Complexity of Sampling},
  author={Liu, Jingbo},
  journal={arXiv preprint arXiv:2605.27769},
  year={2026}
}

@book{watson1944bessel,
  author    = {Watson, G. N.},
  title     = {A Treatise on the Theory of Bessel Functions},
  edition   = {2nd},
  publisher = {Cambridge University Press},
  year      = {1944}
}

@article {bachoc2009optimality,
    AUTHOR = {Bachoc, Christine and Vallentin, Frank},
     TITLE = {Optimality and uniqueness of the {$(4,10,1/6)$} spherical
              code},
   JOURNAL = {Journal of Combinatorial Theory. Series A},
  FJOURNAL = {Journal of Combinatorial Theory. Series A},
    VOLUME = {116},
      YEAR = {2009},
    NUMBER = {1},
     PAGES = {195--204},
      ISSN = {0097-3165,1096-0899},
   MRCLASS = {94B25},
  MRNUMBER = {2469257},
       DOI = {10.1016/j.jcta.2008.05.001},
       URL = {https://doi.org/10.1016/j.jcta.2008.05.001},
}

@article{neuman2004inequalities,
  title={Inequalities involving {B}essel functions of the first kind},
  author={Neuman, Edward},
  journal={Journal of Inequality in Pure and Applied Mathematics},
  volume={5},
  number={4},
  year={2004}
}
